\documentclass[12pt,a4paper]{amsart}
\usepackage[english]{babel}
\usepackage{amsmath}
\usepackage{amsfonts}
\usepackage{amssymb}
\usepackage{amsthm}
\usepackage{graphicx}
\usepackage{enumitem}
\usepackage{stmaryrd}
\usepackage{hyperref}
\usepackage{bbm}
\usepackage{xcolor,mathtools}

\theoremstyle{definition}
\newtheorem{defi}{Definition}[section]
\newtheorem{ex}[defi]{Example}

\newtheorem{rem}[defi]{Remark}
\theoremstyle{plain}
\newtheorem{teo}[defi]{Theorem}
\newtheorem{lem}[defi]{Lemma}
\newtheorem{pro}[defi]{Proposition}

\usepackage[left=2cm,right=2cm,top=2cm,bottom=2cm]{geometry}
\newcommand{\N}{\mathbb{N}}
\newcommand{\he}{\mathbb{H}}
\newcommand{\hel}{\mathfrak{h}}
\newcommand{\vp}{\varphi}
\newcommand{\R}{\mathbb{R}}
\newcommand{\s}{\mathcal{S}}
\newcommand{\D}{\mathcal{D}}
\newcommand{\pp}{\partial}

\newcommand{\V}{\mathbb{V}}
\newcommand{\W}{\mathbb{W}}
\newcommand{\curr}[1]{\llbracket{#1}\rrbracket}

\newcommand{\spa}{\operatorname{span}}
\newcommand{\ve}{\varepsilon}
\newcommand{\bwl}{\text{\Large$\wedge$}}
\newcommand{\col}{\operatorname{col}}
\newcommand{\Id}{\operatorname{Id}}
\newcommand{\se}{\subseteq}

\newcommand{\ceq}{\coloneqq}
\newcommand{\Tcurr}{\mathsf T}
\newcommand{\st}{\text{ s.t. }}
\newcommand{\uno}{\mathbbm{1}}
\renewcommand{\H}{\mathbb{H}}
\newcommand{\res} {\mathop{\hbox{\vrule height 7pt width .5pt depth 0pt \vrule height .5pt width 6pt depth 0pt}}\nolimits}



\begin{document}
\title{Submanifolds with boundary and Stokes' Theorem in Heisenberg groups}

\author{Marco Di Marco}
\address{Dipartimento di Matematica ``T. Levi-Civita'', via Trieste 63, 35121 Padova, Italy.}
\email{marco.dimarco@phd.unipd.it}

\author{Antoine Julia}
\address{Université Paris-Saclay, CNRS, Laboratoire de mathématiques d’Orsay, 91405, Orsay, France.}
\email{antoine.julia@universite-paris-saclay.fr}

\author{Sebastiano Nicolussi Golo}
\address{Tempio e Monastero Zen Sōtō Shōbōzan Fudenji, Salsomaggiore Terme, PR, Italy.}
\email{sebastiano2.72@gmail.com}

\author{Davide Vittone}
\address{Dipartimento di Matematica ``T. Levi-Civita'', via Trieste 63, 35121 Padova, Italy.\newline
\indent School of Mathematics, Institute for Advanced Study, 1 Einstein Drive, 08540 Princeton (NJ), USA.}
\email{davide.vittone@unipd.it}

\thanks{M.~D.~M. and D.~V. are supported by University of Padova and  GNAMPA of INdAM. D.~V. is also supported by PRIN 2022PJ9EFL ``Geometric Measure Theory: Structure of Singular Measures, Regularity Theory and Applications in the Calculus of Variations''. Part of this work was written when D.V.~was a member of the Institute for Advanced Study in Princeton: he wishes to thank the Institute for the support as well as  for for the pleasant and exceptionally stimulating  atmosphere.}

\begin{abstract}
We introduce and study the notion of $C^1_\H$-regular submanifold with boundary in sub-Riemannian Heisenberg groups. As an application, we prove a version of Stokes' Theorem for $C^1_\H$-regular  submanifolds with boundary that takes into account Rumin's complex of differential forms in Heisenberg groups.
\end{abstract}

\maketitle

\section{Introduction}
The classical Stokes' Theorem states that, if $M$ is a smooth $m$-dimensional manifold with boundary, then the equality 
\[
\int_M d\omega=\int_{\pp M} \omega
\]
holds for every smooth $(m-1)$-form on $M$. The aim of the present paper is to provide a generalization of Stokes' Theorem to sub-Riemannian Heisenberg groups, as we now explain. 

Recall that sub-Riemannian Geometry deals with the study of manifolds endowed with a distinguished sub-bundle, called  horizontal bundle, of the tangent bundle. A great deal of effort was devoted in the last decades to the study of {\em intrinsic} analytic and geometric properties of sub-Riemannian manifolds: this is even more true for  {\em Heisenberg groups} which, are the simplest examples of (non-Riemannian) sub-Riemannian manifolds and show such rich features and raise so many interesting open questions that they naturally became subject of extensive investigations.  Heisenberg groups are introduced later in Section~\ref{subsec_Heisenberg}; here we only recall that the $n$-th Heisenberg group $\H^n$ is a step 2 nilpotent Lie group diffeomorphic to $\R^{2n+1}$ that is endowed with a left-invariant, bracket-generating horizontal sub-bundle $H\H^n$ of dimension $2n$, with a one-parameter family $(\delta_\lambda)_{\lambda>0}$ of group isomorphisms (called dilations), and with a left-invariant and homogeneous  distance $d$. The Hausdorff dimension of the metric space $(\H^n,d)$ is $Q\ceq 2n+2$.

In this paper we will prove that Stokes' Theorem holds for submanifolds of Heisenberg groups: before stating it, we however need to discuss the two keywords, {\em submanifolds} and {\em forms} in Heisenberg groups, that are needed for the very statement of Stokes' Theorem.  

The first ingredient is provided by  {\em $C^1_\H$-regular} (or {\em $\H$-regular})  submanifolds, i.e.,  submanifolds that are $C^1$ regular from the viewpoint of intrinsic geometry. 
These submanifolds are of a completely different nature depending on whether they have {\em low dimension} $1\leq k\leq n$, in which case they are $k$-dimensional submanifolds of class $C^1$ tangent to $H\H^n$, or  {\em low codimension} $1\leq k\leq n$, then they are (locally) defined as non-critical level sets of maps from $\H^n$ to $\R^k$ that are continuously differentiable along  horizontal directions. 
In particular, $C^1_\H$-regular submanifolds of low dimension $k$ are (contained in) Legendrian\footnote{For the standard contact structure on the Heisenberg group.} submanifolds and they have classical $C^1$ regularity; their Hausdorff dimension coincides with the topological dimension, $k$. On the contrary, $C^1_\H$-regular submanifolds of low codimension $k$ can have fractal Euclidean  dimension,  while their Hausdorff dimension is $Q-k=2n+2-k$, thus larger than the topological dimension $2n+1-k$. Extensive studies of $C^1_\H$-regular submanifolds have been carried on, see e.g.~\cite{ASCV,ArenaSerapioni,BSC2,BSC1,BV,CittiMan,Corni1,DDFO,franchi}.  We collect in Section~\ref{subsec_C1H} some basic facts about $C^1_\H$-regular submanifolds, but we suggest to the reader the beautiful paper~\cite{franchi}, where the genesis of and motivations behind the theory are masterfully illustrated. 

The second object we need is the complex of {\em Heisenberg differential forms} 
\begin{equation}\label{eq_complessoRumin}
0 \to \mathbb{R} \to \Omega_\he^0 \xrightarrow{d}\Omega_\he^1 \xrightarrow{d}\cdots \xrightarrow{d}\Omega^n_\he \xrightarrow{D}\Omega^{n+1}_\H \xrightarrow{d}\cdots \xrightarrow{d} \Omega^{2n+1}_\he \to 0
\end{equation}
introduced by M.~Rumin~\cite{RuminCR,rumin} (see also~\cite{FischerTripaldi,LerarioTripaldi,tripaldi2020rumin}). For clarity of exposition, in this introduction we will regard the space $\Omega_\he^k$ of Heisenberg $k$-forms as a subspace of classical differential $k$-forms: this is however a simplification and we postpone precise definitions to  Section~\ref{subsec_Rumin}. The operators $d$ appearing in~\eqref{eq_complessoRumin} are restriction of the standard exterior derivative, while the operator $D:\Omega^{n}_\H\to\Omega^{n+1}_\H$ is a non-trivial second-order operator that constitutes  M.~Rumin's key discovery, as it ensures that the  complex~\eqref{eq_complessoRumin} has the same cohomology of the De Rham complex. This complex arises naturally in Heisenberg groups and this paper can be considered as one of the several evidences that Rumin's complex is natural in this setting.

An Heisenberg form $\omega$ can be integrated on an {\em oriented} (see Definition~\ref{defi_orientabilecodimbassa}) $C^1_\H$-regular submanifold $S$. If $S$ is of low dimension then $\int_S\omega$  is  the standard integral (see, however, Proposition~\ref{prop_integrazionebassa}); if, on the other hand, $S$ has low codimension we set
\begin{equation}\label{eq_defintegration}
\int_S\omega \ceq \frac{1}{C_{n,k}}\int_S \langle t^\he_S|\omega   \rangle d\s^{Q-k},
\end{equation}
where $ t^\he_S$ denotes the Heisenberg  unit tangent vector orienting $S$, $\s^{Q-k}$ is the spherical Hausdorff measure of dimension $Q-k$, and $C_{n,k}$ is a geometric constant with the property that $\int_S\omega$ coincides with the standard integral when $S$ is Euclidean smooth, see Lemma~\ref{eqcurr}. The operator $\omega\mapsto\int_S\omega$ is a {\em Heisenberg current}, which we denote by $\curr S$. The theory of currents in Heisenberg groups (introduced in Section~\ref{subsec_correnti}) is only at a very early stage~\cite{CanarecciCurrents,franchi,JNGVIMRN,vittone}, but the formalism is quite convenient for our purposes and, will be extensively used later in the paper.
It is worth noticing that the value of $\int_S\omega$ does not depend on the particular choice of the distance $d$, see Lemma~\ref{lem_stessovalore}: in particular, the definition in~\eqref{eq_defintegration} seems to us natural and substantial -- in a word, {\em geometric}.

The first contribution of the present paper is the study of $C^1_\H$-regular submanifolds with boundary in Heisenberg groups, see also \cite[Section 5]{pansu}. One could be tempted to consider as such those $C^1_\H$-regular submanifolds $S$ whose boundary, defined as $\pp S \ceq \overline S\setminus S$, is also a $C^1_\H$-regular submanifold; however, this definition would not be appropriate. To explain why, we use the following example in Euclidean geometry (similar examples could be produced in Heisenberg groups): the surface $S=\{(x,y,x\sin\frac 1x):x>0,y\in\R\}$ is $C^\infty$ smooth in $\R^3$ and its boundary $\pp S=\{(0,y,0):y\in\R\}$ is also smooth; however, $S$ is not of class $C^1$ as a submanifold {\em with boundary}. The classical definition of $C^1$ manifold with boundary $S$ requires that, for every  $p\in \pp S$, there is a $C^1$ diffeomorphism $\Phi$ between a neighbourhood of $p$ and a closed half-ball such that $\Phi(p)$ is the center of the ball; using a little bit of Whitney Extension Theorem, one realizes that this is equivalent to requiring that, for every $p\in\pp S$, the manifold $S$ can be (locally) extended to a $C^1$ submanifold $S'$ whose interior contains $p$. 

This discussion motivates the following definition.

\begin{defi}\label{def_supboundintro}
Let $1 \leq m \leq 2n+1$; we say that $S \se \H^n$ is an \emph{$m$-dimensional $C^1_\he$-regular submanifold with boundary} if the following conditions hold:
\begin{enumerate}[label=(\arabic*)]
\item $ S$ is a non empty $m$-dimensional $C^1_\he$-regular submanifold;
\item $ \pp S$ is a non empty $(m-1)$-dimensional $C^1_\he$-regular submanifold;
\item for every $p \in  \pp S$ there exist a neighbourhood $U$ of $p$ and an  $m$-dimensional $C^1_\H$-regular submanifold $S' \se U$ such that $U \cap \overline{S} \se S'$ and, for every $r>0$,  $U(p,r) \cap (S'\setminus \overline{S})\neq \emptyset$.
\end{enumerate}
\end{defi}

In Definition~\ref{def_supbound} we also introduce the (local) notion of submanifolds with boundary that are $C^1_\H$-regular in an open set. Definition~\ref{def_supboundintro}  is stated to work  for any dimension and codimension; however, it is quite clear that in the low-dimensional case $C^1_\he$ regular submanifolds with boundary are exactly Euclidean $C^1$ submanifolds with boundary that are tangent to the horizontal distribution.\footnote{This implies that also the boundary, whose tangent space is contained in the tangent space to the submanifold, is tangent to the horizontal distribution.} In the case of low codimension $k=2n+1-m$ one needs to distinguish between the non-critical case $1\leq k\leq n-1$ and the critical case $k=n$: the first is somewhat easier and one can provide a useful equivalent definition in terms of level/superlevel sets, see Theorem~\ref{teo_equivsupbound}. On the contrary, the case of critical codimension $k=n$ is much more delicate, as it corresponds to the case in which the submanifold is of low codimension  but its boundary is low dimensional. This is why  in the critical case it is much more difficult to prove our main result, which we now state.

\begin{teo}
\label{thm_intro}
If $1 \leq m \leq 2n+1$, $S \subset \H^n$ is an $m$-dimensional orientable $C^1_\he$ regular submanifold with boundary and $S,\pp S$ have locally finite  measures in $\H^n$, then
\begin{equation}\label{eq_enunciatoStokes}
\int_S d_c\omega = \int_{\pp S} \omega \qquad\text{for every }\omega\in\D_\H^{m-1},
\end{equation}
where $\D^{m-1}_\H$ denotes the space of Heisenberg differential $(m-1)$-forms with compact support.
\end{teo}

As customary, the symbol $d_c$ stands for the exterior Rumin differential: $d_c=d$ if $m\neq n+1$, while in the critical case $m=n+1$ one writes $d_c=D$. 
The measures on $S,\pp S$ referred to in Theorem~\ref{thm_intro} are the Hausdorff spherical measures on $S,\pp S$ of the appropriate dimensions that are (respectively) $m,m-1$ (in the low dimensional case $m\leq n$), $m+1,m$ (in the low-codimensional and non-critical case $m\geq n+2$) or $n+2,n$ (in the critical case $m=n+1$). 
A  local version of Theorem~\ref{thm_intro} is also proved in Theorem~\ref{fin1}, from which Theorem~\ref{thm_intro} follows. Observe that Stokes' Theorem can be written in the language of currents as $\pp_c\curr S=\curr{\pp S}$, and  Theorem~\ref{fin1} is indeed stated in these terms. 

In the low-dimen\-sio\-nal case, Theorems~\ref{thm_intro} and~\ref{fin1} follows directly from the classical Stokes' Theorem. For $C^1_\H$-regular submanifolds of low codimension $1\leq k\leq n$ they can be proved via an approximation scheme, performed in Section~\ref{sec_approssimazione}, in which the submanifold $S$ and its boundary $\pp S$ are $C^1_\H$-approximated by, respectively, smooth submanifolds $S_h$ and their boundaries $\pp S_h$: the classical Stokes Theorem holds for each $S_h$ and the equality $\int_{S_h} d_c\omega = \int_{\pp S_h} \omega$ passes to the limit, as $h\to\infty$, to get~\eqref{eq_enunciatoStokes}. The approximation scheme is particularly delicate in the case of critical codimension $k=n$: the major difficulty one encounters is that, while it is fairly easy to produce smooth approximations $S_h$ of $S$, Frobenius Theorem makes it quite challenging to ``cut'' into $S_h$ an $n$-dimensional boundary $\pp S_h$ that is tangent to the horizontal distribution, as the latter is totally non-integrable. Note that this difficulty does not arise if $k=n=1$ in which case it is enough to find some approximation $S_h$  and then  ``cut'' it (thus creating the boundary $\pp S_h$) along some horizontal curve. This idea was exploited in~\cite[Theorem 5.3]{franchitris}, where  Stokes' Theorem in $\H^1$ was proved for the critical dimension, although with a different formulation. However, following a technically quite elaborate argument, in  Lemma~\ref{approx2_crit} we were able to produce an approximating sequence $S_h$ with  the extra property that $\pp S_h=\pp S$ for every $h$, which is even more than requested. We believe that the approximation results of Lemmata~\ref{approx2} and~\ref{approx2_crit} might be of use for future applications.

\section{Notation and preliminary results}
\subsection{Heisenberg groups and Heisenberg algebras}\label{subsec_Heisenberg}
\begin{defi}
For $n\geq 1$ we denote by $\he^n$ the $n$-th {\em Heisenberg group}, identified with $\R^{2n+1}$ through exponential coordinates. We denote a point $\he^n \ni p=(x,y,t)$ by $x,y \in \R^{n}$ and $t \in \R$. If $p=(x,y,t),q=(x',y',t') \in \he^n$, the group operation is defined as 
\[
p \cdot q \ceq (x+x',y+y',t+t'+\tfrac{1}{2}\langle x,y' \rangle-\tfrac{1}{2}\langle x', y \rangle).
\]
If $\he^n \ni p=(x,y,t)$, its inverse is $p^{-1}=(-x,-y,-t)$ and $0=(0,0,0)\in \he^n$ is the identity of $\he^n$. 

Let $p,q \in \he^n$, we denote by $\tau_p:\he^n \to \he^n$ the left translation, i.e., $\tau_p(q) \ceq p\cdot q$. For $\lambda >0$, we denote by $\delta_\lambda:\he^n\to\he^n$ the dilations of the Heisenberg groups defined, for $p=(x,y,t)\in\he^n$, by $\delta_\lambda(x,y,t) \ceq (\lambda x,\lambda y, \lambda^2 t)$. Observe that dilations form a one-parameter family of group isomorphisms.

We denote by $Q \ceq 2n+2$ the {\em homogeneous dimension} of $\he^n$. The Lebesgue measure $\mathcal L^{2n+1}$ is the Haar measure on $\he^n\equiv\R^{2n+1}$ and it is $Q$-homogeneous with respect to dilations.
\end{defi}

\begin{defi}
We denote by $\hel^n$ (or by $\hel$ when the dimension $n$ is clear) the $(2n+1)$-dimensional Lie algebra of  left invariant vector fields in $\he^n$. The algebra $\hel$ is generated by the vector fields $X_1,...,X_n,Y_1,...,Y_n,T$ where (for $1 \leq j \leq n$)
\[
X_j \ceq \pp_{x_j}-\frac{y_j}{2}\pp_t,
\qquad
Y_j \ceq \pp_{y_j}+\frac{x_j}{2}\pp_t,
\qquad
T \ceq \pp_t.
\]
We denote by $\hel_1$ the horizontal subspace of $\hel$, i.e.,
\[
\hel_1 \ceq \operatorname{span}(X_1,...,X_n,Y_1,...,Y_n),
\]
and by $\hel_2$ the linear span of $T$; the Lie algebra $\hel$ admits the 2-step stratification $\hel=\hel_1 \oplus \hel_2$. 

If $p \in \he^n$ we also denote by $H_p\he^n \ceq \operatorname{span}(X_1(p),...,X_n(p),Y_1(p),...,Y_n(p))$ and by $H \he^n \ceq \bigcup_{p \in \he^n}H_p\he^n$ the horizontal fiber bundle. We denote by $\langle \cdot, \cdot \rangle$ the inner product on $\hel$ that makes the basis $X_1,..,X_n,Y_1,...,Y_n,T$ orthonormal and if $p \in \he^n$ we denote by $\langle \cdot, \cdot \rangle_p$ the corresponding inner product on $H_p \he^n$ that makes the basis $X_1(p),...,X_n(p),Y_1(p),...,Y_n(p)$ orthonormal. We will use $|\cdot|$ and $|\cdot|_p$ to denote the corresponding norms.
\end{defi}

\begin{defi}\label{def_multialgebre}
We define the $k$-th exterior algebra of $k$-vectors and $k$-covectors as, respectively,
\begin{align*}
&\bwl_k \hel \ceq \operatorname{span}\lbrace W_{i_{1}}\wedge \cdots \wedge W_{i_k} \mbox{ with }1 \leq i_1 \leq \cdots \leq i_k \leq 2n+1 \rbrace\\
&\bwl^k \hel \ceq \operatorname{span}\lbrace \theta_{i_{1}}\wedge \cdots \wedge \theta_{i_k} \mbox{ with }1 \leq i_1 \leq \cdots \leq i_k \leq 2n+1 \rbrace
\end{align*}
where we used the notation $W_i \ceq X_i$ if $1 \leq i \leq n$, $W_{i} \ceq Y_{i-n}$ if $n+1 \leq i \leq 2n$ and $W_{2n+1} \ceq T$. Similarly we define $\theta_i \ceq dx_i$ if $1 \leq i \leq n$, $\theta_{i} \ceq dy_{i-n}$ if $n+1 \leq i \leq 2n$ and $\theta_{2n+1}\ceq \theta \ceq dt+\frac{1}{2}\sum_{j=1}^n(y_jdx_j-x_jdy_j)$; observe that $\lbrace \theta_i,...,\theta_{2n},\theta \rbrace$ is the dual basis to $\lbrace X_1,...,Y_n,T \rbrace$. In the same fashion we define the $k$-th exterior algebra of horizontal $k$-vectors and horizontal $k$-covectors as 
\begin{align*}
&\bwl_k \hel_1 \ceq \operatorname{span}\lbrace W_{i_{1}}\wedge \cdots \wedge W_{i_k} \mbox{ with }1 \leq i_1 \leq \cdots \leq i_k \leq 2n\rbrace
\\
&\bwl^k \hel_1 \ceq \operatorname{span}\lbrace \theta_{i_{1}}\wedge \cdots \wedge \theta_{i_k} \mbox{ with }1 \leq i_1 \leq \cdots \leq i_k \leq 2n \rbrace.
\end{align*}
The (canonical) action of a $k$-covector $\omega$ on a $k$-vector $v$ is denoted by $\langle v | \omega \rangle$. The fiber of $\bwl_k \hel$ over $p \in \he^n$ is denoted by $\bwl_{k,p}\hel$ and analogously for $\bwl_{k,p}\hel_1$. The inner product $\langle \cdot, \cdot \rangle$ defined on $\hel$ extends canonically to $\bwl_k \hel$ and $\bwl^k \hel$ (and then, by restriction, to $\bwl_k \hel_1$ and $\bwl^k \hel_1$). In the following we drop the subscripts in $\langle \cdot, \cdot \rangle_{\bwl_k \hel}$ and $| \cdot |_{\bwl_k \hel}$.
\end{defi}

\begin{defi} 
We say that a function $d:\H^n \times \H^n \to [0,+\infty)$ is a \emph{left invariant and homogeneous distance} if
\begin{enumerate}[label=(\roman*)]
\item $d(p,q)=d(r \cdot p,r \cdot q)$ for all $p,q,r \in \H^n$,
\item $d(\delta_\lambda(p),\delta_\lambda(q))=\lambda d(p,q)$ for all $p,q \in \H^n$ and $\lambda>0$.
\end{enumerate}
We define the associated norm $\|\cdot \|$ to $d$ as $\|p\|\ceq d(0,p)$ for every $p \in \H^n$. Moreover, if for every $(x,y,t),(x',y',t) \in \H^n$ such that $|(x,y)|_{\R^{2n}}=|(x',y')|_{\R^{2n}}$ we have $\|(x,y,t)\|=\|(x',y',t)\|$ we say that $d$ is a \emph{rotationally invariant} distance.

For $r>0$ and $p \in \he^n$ we define the \emph{open ball} $U(p,r)$ as
\[
U(p,r) \ceq \lbrace q \in \he^n : d(p,q)<r \rbrace.
\]
\end{defi}
In the following we denote with $d$ a fixed left invariant, homogeneous and rotationally invariant distance on $\H^n$ and with $U(\cdot,\cdot)$ the associated open balls.
\begin{ex}\label{ex_dist}
There are many examples of left invariant, homogeneous and rotationally invariant distances on $\H^n$, the most notable being the following.
\begin{enumerate}[label=(\roman*)]
\item The \emph{Carnot-Carathéodory} distance $d_{cc}$ defined for $p \in \H^n$ as
\[
d_{cc}(0,p):=\inf\left\{\|h\|_{L^1([0,1],\R^{2n})} : 
\begin{array}{l}
\text{the curve $\gamma_h:[0,1]\to\H^n$ defined by}\\
\gamma_h(0)=0,\ \dot\gamma_h=\sum_{j=1}^n(h_jX_j+h_{j+n}Y_j)(\gamma_h)\\
\text{has final point }\gamma_h(1)=p
\end{array}
\right\}.
\]
\item The \emph{infinity} distance $d_\infty$ defined for $(x,y,t) \in \H^n$ as
\[
d_\infty(0,(x,y,t)) \ceq \max \lbrace |(x,y)|_{\R^{2n}},2|t|_\R^{\frac{1}{2}}\rbrace.
\]
\item The \emph{Korányi (or Cygan-Korányi)} distance $d_K$ defined for $(x,y,t) \in \H^n$ as
\[
d_K(0,(x,y,t)) \ceq \left((|x|_{\R^n}^2+|y|_{\R^n}^2)^2+16t^2 \right)^\frac{ 1}{4}.
\]
\end{enumerate}
\end{ex}
\begin{pro}[{\cite[Proposition 1.3.15]{didonatophd}}]\label{pro_distequiv}
Let $d_1$ and $d_2$ be left invariant and homogeneous distances on $\H^n$. Then they are bi-Lipschitz equivalent, i.e., there exists $C>0$ such that for all $p,q \in \H^n$ 
\[
\frac{1}{C}d_2(p,q)\leq d_1(p,q) \leq C d_2(p,q).
\]
In particular every left invariant and homogeneous distance induces the Euclidean topology on $\H^n$.
\end{pro}
\begin{defi}
Let $m \geq 0$. Given any left invariant, homogeneous and rotationally invariant distance on $\H^n$ we denote by $\s^m$ the spherical Hausdorff measure on $\he^n$ defined for $E\subset \he^n$ by
\[
\s^m(E) \ceq \lim_{r\to0^+} \inf\left\{ \sum_{i\in \N} (2r_i)^m :\exists \;(p_i)_{i}\subset\he^n,\exists\;(r_i)_{i}\text{ with }0<r_i<r\text{ and }E\subset\bigcup_{i\in \N}U(p_i,r_i)  \right\}.
\]
\end{defi}

It is well known that $\s^Q$ coincides with $\mathcal L^{2n+1}$ up to a positive multiplicative constant; in particular, $Q$ is the Hausdorff dimension of $\he^n$.

\begin{defi}[Projection on the ``horizontal'' hyperplanes]
\label{proj}
Let $q=(x_1,...,x_n,y_1,...,y_n,t) \in \he^n, p \in \he^n$ be given. We set 
\[
\pi_p(q) \ceq \sum_{j=1}^n x_j X_j(p)+\sum_{j=1}^n y_j Y_j(p).
\]
For fixed $p$, the map $\pi_p$ is a projection, while for fixed $q$, the map $p \to \pi_p(q)$ is a smooth section of the horizontal bundle $H \he^n$ (it is in fact a left-invariant vector field).
\end{defi}

\begin{pro}
For any $a,b,p \in \he^n$ the following relations hold:
\begin{enumerate}
\item[(1)] $\pi_p(a \cdot b)=\pi_p(a)+\pi_p(b)$;
\item[(2)] $d(a,b)\geq C |\pi_p(a)-\pi_p(b)|_p$.
\end{enumerate}
where $C$ is a positive constant only depending on the distance $d$.
\end{pro}
We omit the boring proof.

\subsection{\texorpdfstring{$C^1_\he$}{C1H}-regular functions and \texorpdfstring{$C^1_\he$}{C1H}-regular submanifolds}\label{subsec_C1H}
\begin{defi}
Let $U \subset \he^n$ be an open set and $f:U \to \R^k$. We say that $f$ is {\em P-differentiable} at $p \in U$ if there is a (necessarily unique) group homomorphism $d_\he f_{p}:\he^n \to \R^k$ such that
\[
d_\he f_{p}(q) \ceq \lim_{\lambda \to 0}\frac{f(p \cdot \delta_\lambda(q))-f(p)}{\lambda}\qquad\forall\; q\in\he^n
\]
locally uniformly in $q$.

Let $U \subset \R^k$ be an open set and $f:U \to \he^n$. We say that $f$ is P-differentiable at $a \in U$ if there is a (necessarily unique) group homomorphism $d_\he f_{a}:  \R^k \to \he^n$ such that
\[
d_\he f_{a}(b) \ceq \lim_{\lambda \to 0} \delta_{1/\lambda} \left( f(a)^{-1}\cdot f(a+\lambda b) \right)\qquad\forall\; b\in\R^k
\]
locally uniformly in $b$.
\end{defi}

\begin{defi}
Let $U \subset \he^n$ be an open set and $f:U \to \R$; we say that $f$ is {\em of class $C^1_\he$} ($f \in C^1_\he(U)$) if $f$ is continuous and its {\em horizontal gradient} $\nabla_\he f \ceq (X_1f,...,X_nf,Y_1f,...,Y_nf)$ (to be understood in the sense of distributions) is represented by a $2n$-ple of continuous functions on $U$. 

We agree that, for every $p \in U$, the horizontal gradient $\nabla_\he f(p)$ is identified with the horizontal vector
\[
\nabla_\he f(p)=X_1f(p)X_1(p)+\cdots+Y_nf(p)Y_n(p)\in H_p\he^n.
\]
\end{defi}

\begin{rem}
It is well-known (see e.g.~\cite[Proposition~2.4]{jnv}) that, if $f\in C^1_\he(U)$, then $f$ is P-differentiable at every $p\in U$ and
\[
d_\he f_p(q)= \sum_{i=1}^n x_i X_if(p) + y_i Y_if(p)\qquad\forall\;q=(x,y,t)\in\he^n.  
\]
\end{rem}

\begin{defi}
\label{defsup2}
Let $1 \leq k \leq n$. A subset $S \subset \he^n$ is a $k$-dimensional (or $(2n+1-k)$-codimensional) {\em $C^1_\he$-regular} (or {\em $\he$-regular}) \emph{submanifold} 
if for any $p \in S$ there are open sets $U \subset \he^n$, $V \subset \R^k$ a and a function $f:V \to U$ such that $p \in U$, $f$ is injective, $f$ is continuously P-differentiable with $d_\he f$ injective and
\[
S \cap U =f(V).
\] 
\end{defi}

\begin{rem}\label{rem_C1HC1dimensionebassa}
By \cite[Theorem 3.5]{franchi}, a $C^1_\he$ regular submanifold $S$ of  dimension $k\leq n$ is also a Euclidean $C^1$ submanifold and $T_pS\subset H_p\he^n$ for every $p\in S$. Moreover, the Hausdorff dimension of $S$ equals the topological dimension  $k$ and the spherical Hausdorff measure $\s^k\res S$ is comparable with the Euclidean $k$-dimensional Hausdorff measure on $S$.
\end{rem}

The following lemma will be useful later.

\begin{lem}\label{lem_distlowdim}
Let $1 \leq k \leq n$ and $S \subset \he^n$ be a $k$-dimensional $C^1_\he$ regular submanifold. Then for any $p \in S$ there exist an open set $U \subset \he^n$ with $p \in U$ and a positive constant $C$ such that for any $q,q' \in U \cap S$ one has
\[
d(q,q') \leq C|\pi(q)-\pi(q')|_{\R^{2n}}
\]
where $\pi:\he^n \to \R^{2n}$ is the projection map defined by $\pi(x,y,t) \ceq (x,y)$.
\end{lem}
\begin{proof} 
By Proposition \ref{pro_distequiv} it is sufficient to prove the result when $d=d_\infty$ (as defined in Example \ref{ex_dist}). Fix $p \in S$. Let $V \se \R^k$ and $O \se \he^n$ be open sets such that $p \in O$ and $f=(f_1,...,f_{2n+1}):V \to \he^n$ be the defining function of $S$ coming from Definition \ref{defsup2}, i.e., $f(V)=O \cap S$. Let $U=U(p,r)$ where $r>0$ will be chosen later. Let $q,q' \in U \cap S$ and $a,b,b' \in V$ such that $f(a)=p$, $f(b)=q$ and $f(b')=q'$. One has
\begin{equation}\label{eq_p1}
\begin{aligned}
d(q,q')&=d(f(b),f(b'))=d((f(b))^{-1}\cdot f(b'),0) \\&\leq d(d_\he f_a (b-b'),(f(b))^{-1}\cdot f(b'))+d(d_\he f_a (b-b'),0).
\end{aligned}
\end{equation}
From \cite[Proposition~2.4]{jnv} we obtain
\[
d(d_\he f_a(b-b'),(f(b))^{-1}\cdot f(b')) \leq |b-b'|_{\R^k}
\]
and from \cite[Proposition~2.7]{franchi} one has (since the last component of $d_\he f_a$ is always 0)
\[
d(d_\he f_a (b-b'),0)=|d_\he f_a(b-b')|_{\R^{2n+1}}.
\]
Using this inequalities in \eqref{eq_p1} together with the fact that, by the injectivity of $d_\he f_a$, there exists a positive constant $C$ such that $|b-b'|_{\R^{k}}\leq C|d_\he f_a (b-b')|_{\R^{2n}}$ we get
\[
d(q,q') \leq (1+C)|d_\he f_a(b-b')|_{\R^{2n+1}}.
\]
By \cite[Proposition~2.7]{franchi} we have
\[
d_\he f_a(b-b')=\left( \begin{array}{c}
\langle \nabla f_1 (a),b-b'\rangle_{\R^k}\\
\vdots\\
\langle \nabla f_{2n} (a),b-b'\rangle_{\R^k}\\
0\\
\end{array}\right).
\]
We claim that there exists a positive constant $C$ such that for every $1 \leq j \leq 2n$ 
\[
|\langle \nabla f_j (a),b-b' \rangle_{\R^k}| \leq C|f_j(b)-f_j(b')|;
\]
this would prove the Lemma. We have (without loss of generality we can suppose $b \neq b'$)
\begin{equation}\label{eq_p2}
|\langle \nabla f_j (a),b-b' \rangle_{\R^k}|=|b-b'|_{\R^k}\left| \left\langle \nabla f_j(a),\frac{b-b'}{|b-b'|_{\R^k}}\right\rangle_{\R^k} \right|=|b-b'|_{\R^k}\left| \frac{\pp f_j}{\pp v}(a)\right|
\end{equation}
where we used $\frac{\pp f_j}{\pp v}$ to denote the derivative of $f_j$ in the direction $v=\frac{b-b'}{|b-b'|_{\R^k}}$. By definition we have
\[
\frac{\pp f_j}{\pp v}(a)=\lim_{b,b' \to a} \frac{f_j(b)-f_j(b')}{b-b'};
\]
therefore, possibly reducing $V$, we can assume that 
\[
\left|\frac{\pp f_j}{\pp v}(a)\right| \leq  1+\left| \frac{f_j(b)-f_j(b')}{b-b'}\right|\quad\forall\;b,b'\in V.
\]
If we use this information in \eqref{eq_p2} we obtain
\begin{equation}\label{eq_p3}
|\langle \nabla f_j (a),b-b' \rangle_{\R^k}| \leq  |b-b'|_{\R^k}+|f_j(b)-f_j(b')|
\end{equation}
for every $b,b' \in V$. From \cite[Theorem 3.5]{franchi} $f_j$ is $C^1$ and injective with injective (classical) differential $df_a$ so, possibly reducing $V$, $f_j$ is biLipschitz continuous, i.e., there exist  $K>0$ such that for any $b,b' \in V$ one has
\[
\frac{1}{K}|b-b'|_{\R^k}\leq |f_j(b)-f_j(b')|\leq K|b-b'|_{\R^k}.
\]
From the latter and \eqref{eq_p3} we eventually obtain
\[
|\langle \nabla f_j (a),b-b' \rangle_{\R^k}| \leq  |b-b'|_{\R^k}+|f_j(b)-f_j(b')| \leq (K+1)|f_j(b)-f_j(b')|
\]
for any $b,b' \in V $. The proof is concluded provided we choose $U=U(p,r)$ with $r>0$ small enough such that $U(p,r) \cap S \se f( V)$.
\end{proof}
\begin{defi}
\label{defsup}
Let $1 \leq k \leq n$. A subset $S \subset \he^n$ is a $k$-codimensional (or $(2n+1-k)$-dimensional) {\em $C^1_\he$ regular} (or {\em $\he$-regular}) \emph{submanifold}  if for any $p \in S$ there are an open set $U \subset \he^n$ and a function $f:U \to \R^k$ such that $p \in U$, $f=(f_1,\dots,f_k) \in C^1_\he(U,\R^k)$ and
\begin{align*}
& S \cap U =\lbrace q \in U :f(q)=0 \rbrace\\
& \text{$\nabla_\he f_1 \wedge \cdots \wedge \nabla_\he f_k \neq 0$ on $U$ (equivalently, $d_\he f$ is onto).}
\end{align*}
The {\em tangent group} $T^g_\he S(p)$ to $S$ at $p$  is the subgroup of $\he^n$ defined by
\[
T^g_\he S(p) \ceq \ker d_\he f_p.
\] 
The \emph{horizontal normal} at $p \in S$ $n^\he_S(p) \in \bwl_{k,p} \mathfrak{h}_1$ is defined by
\[
n^\he_S(p) \ceq \frac{\nabla_\he f_1(p) \wedge \cdots \wedge \nabla_\he f_k(p)}{|\nabla_\he f_1(p) \wedge \cdots \wedge \nabla_\he f_k(p)|}.
\] 
The \emph{tangent vector} $t^\he_S(p) \in \bwl_{2n+1-k,p}\mathfrak{h}$ is defined by
\[
t^\he_S(p) \ceq *n^\he_S(p)
\]
where $*$ is the Hodge operator. By \cite[Proposition 3.29]{franchi}, $t^\he_S(p)$ and $n^\he_S(p)$ are well defined (up to a sign) and they do not depend on the defining function $f$. The tangent vector is never horizontal and it can always be written in the form
\[
t_S^\he(p)=\tau_S^\he(p) \wedge T
\]
for a unique (up to a sign) $\tau_S^\he(p) \in \bwl_{2n-k,p}\hel_1$. 

Observe that  $n^\he_S(p),t_S^\he(p)$ and $\tau_S^\he(p)$ can be (locally) chosen to be continuous in $p$. Moreover, $T_\he^g S(p)=\spa t^\he_S(p)$.

We define the {\em boundary} $\pp S$ of $S$ as $\pp S \ceq \overline{S} \setminus  S$.
\end{defi}

An equivalent definition of low codimensional $C^1_\he$ regular submanifold is provided in the subsequent Lemma~\ref{fsmooth}, the proof of which requires  group convolution, which we now introduce. See also \cite[Chapter 1]{follandstein}.

\begin{defi}
Let $k$ be a positive integer, $H \in C^\infty_c(\he^n,\R)$ and $G:\he \to \R^k$. We use $G \star H$ to denote the \emph{group convolution} between $G$ and $H$, which is defined, for every $x \in \he^n$, as
\[
(G \star H)(x) \ceq \int_{\he^n}G(y^{-1}\cdot x)H(y)d\mathcal{L}^{2n+1}(y)=\int_{\he^n}G(y)H(x \cdot y^{-1})d \mathcal{L}^{2n+1}(y).
\]
Notice that $G \star H$ is a smooth map satisfying
\[
W(G \star H)=(WG)\star H
\]
for every $W \in \hel$.

Let $\ve>0$. In the following we will use $K_{\ve}$ to denote the \emph{standard mollification kernel} that is, $K_\ve \ceq \ve^{-Q}K \circ \delta_{1/\ve}$, where $K \in C^\infty_c(U(0,1))$ is a fixed non-negative function (a kernel) such that $\int_{\he^n}K d \mathcal{L}^{2n+1}=1$.
\end{defi}

We use the compact notation $\{f=0\}$ to denote the  0 level set $\{p\in D:f(p)=0\}$ of a function $f:D\to\R^k$; the domain $D$ of $f$ will always be clear from the context. Moreover, for $f\in C^1_\he(\he^n,\R^k)$ and $q \in \H^n$ we denote with $\widehat \nabla_\he f(q)$ the $k \times k$ matrix $\col[X_1f|\cdots|X_kf](q)$ and we denote by $\Id_{k \times k}$  the $k\times k$ identity matrix. The following lemma will be useful later.

\begin{lem}
\label{fsmooth} 
Let $1 \leq k \leq n$ and $S\subset \he^n$. The following statements are equivalent:
\begin{enumerate}
\item[(i)] $S$ is a $C^1_\he$ submanifold of codimension $k$;
\item[(ii)] for every $p \in S$ there exist an open set $U$ with $p \in U$, a function $f:\he^n \to \R^k$ and  $\delta_0>0$ such that, up to an isometry of $\he^n$,
\begin{enumerate}
\item[(1)] $f \in C^1_\he(\he^n,\R^k)$,
\item[(2)] $|\nabla_\he f|$ is bounded on $\he^n$,
\item[(3)] $f \in C^\infty(\he^n \setminus \lbrace f=0 \rbrace,\R^k)$,
\item[(4)] $\widehat \nabla_\he f(q)\geq  \delta_0 \Id_{k \times k}$ for every $q \in \he^n$ in the sense of quadratic forms,
\item[(5)] $S \cap U=\lbrace q \in U$ : $f(q)=0\rbrace$.
\end{enumerate}
\end{enumerate}
\end{lem}
\begin{proof}
The implication $(ii)\Rightarrow(i)$ is clear. Let us prove the implication $(i)\Rightarrow(ii)$. Fix $p \in S$ and an open ball $U(p,3r)$ where $r>0$ will be chosen later. Then there exist a function $g \in C^1_\he(U(p,3r),\R^k)$ with $\nabla_\he g$ of maximal rank on $U(p,3r)$ and 
\[
S \cap U(p,3r)=\lbrace  q \in U(p,3r) :g(q)=0 \rbrace.
\]
Using the Whitney Extension Theorem (see \cite[Theorem 6.8]{franchibis} and \cite[Theorem 2.3.8]{didonatophd}), it is not restrictive to assume that $g \in C^1_\he(\he^n,\R^k) \cap C^\infty(\he^n \setminus [\lbrace g=0 \rbrace\cap U(p,3r)],\R^k)$.
By~\cite[Proposition~3.25]{franchi}, up to an isometry of $\he^n$ we can assume that $X_1g(p),\dots,X_k g(p)\in\R^k$ are linearly independent; in particular, there exists $L\in GL(\R^k)$ such that $L(X_ig(p))=e_i$, where $e_1,\dots,e_k$ is the canonical basis of $\R^k$. Upon replacing $g$ with $L\circ g$ we have $\widehat\nabla_\he g(p)= \Id_{k\times k}$. Possibly reducing $r$, we can assume that $|\nabla_\he g|$ is bounded on $U(p,3r)$ and that $\widehat\nabla_\he g\geq \tfrac 34 \Id_{k\times k}$ on $U(p,3r)$.

Fix $\chi_1 \in C^\infty_c (\H^n)$ such that $0 \leq \chi_1 \leq 1$, $\chi_1 \equiv 1$ on $U(p,1)$ and $\chi_1 \equiv 0$ on $\H^n \setminus U(p,2)$; set $C \ceq \| \nabla_\H \chi_1\|_{C^0}$ and $\chi_r(q)\ceq \chi_1(\delta_{1/r}(q))$ for $q \in \H^n$. Then the functions $\chi_r \in C^\infty_c(\he^n)$ satisfy
\begin{equation}\label{condchi}
\begin{cases}
0 \leq \chi_r \leq 1\\
\chi_r \equiv 1 \mbox{ on } U(p,r)\\
\chi_r \equiv 0 \mbox{ on }\he^n  \setminus U(p,2r) \\
|\nabla_\he \chi_r|\leq \frac{C}{r} .
\end{cases}
\end{equation}
We now consider the differential $d_\he g_p:\he^n \to \R^k$ of $g$ at $p$ and define $h:\he^n \to \R^k$ by 
\[
h \ceq \chi_r g+(1-\chi_r )d_\he g_p.
\] 
Clearly, $h \in C^1_\he(\he^n,\R^k)$, $h$ is Lipschitz continuous and $S \cap U(p,r)=\lbrace p \in U(p,r):h(p)=0 \rbrace$.

We claim that
\begin{equation}
\label{eq_1/2}
\text{ $\widehat{\nabla}_\he h\geq \tfrac12\Id_{k\times k}$\qquad on $\he^n$};
\end{equation}
we have
\begin{align*}
\widehat{\nabla}_\he h
& =\widehat{\nabla}_\he \chi_r \otimes (g-d_\he g_p)+\chi_r \widehat{\nabla}_\he g+(1-\chi_r)\widehat{\nabla}_\he g (p)\\
&\geq \widehat{\nabla}_\he \chi_r \otimes (g-d_\he g_p)+\chi_r \tfrac 34\Id_{k\times k}+(1-\chi_r) \Id_{k\times k}\\
&\geq \widehat{\nabla}_\he \chi_r \otimes (g-d_\he g_p)+ \tfrac 34\Id_{k\times k}.
\end{align*}

We are left to estimate the first term $\widehat{\nabla}_\he\chi_r \otimes (g-d_\he g_p)$. We have that
\[
\|\widehat{\nabla}_\he\chi_r \otimes (g-d_\he g_p)\|_{C^0(\he^n)} 
\leq \frac{C}{r}\|g-d_\he g_p\|_{C^0(U(p,2r))}
\leq \frac{C}{r} o(r) = o(1),
\]
where we used the Taylor expansion of $g$ at $p$. This proves the claim~\eqref{eq_1/2} provided $r>0$ is chosen small enough. 

Now we use a standard mollification procedure (as for instance in \cite[Proposition~2.10]{vittone}) to construct a new function  $f$ which satisfies the properties $(1)$-$(5)$ upon setting $U=U(p,r)$. First we define $Z_h \ceq \lbrace p \in \he^n:h(p)=0 \rbrace$. For $j \in \N$ we choose
\begin{itemize}
\item bounded open sets $(U_j)_{j \in \N}$ such that $\overline{U_j}\subset \he^n \setminus Z_h$ and $\he^n \setminus Z_h=\bigcup_j U_j $,
\item positive numbers $\ve_j$ such that $\ve_j<d(U_j,Z_h)$,
\item nonnegative functions $u_j \in C^\infty_c(U_j)$ forming a partition of the unity on $\he^n \setminus Z_h$.
\end{itemize}
We can also assume that $\sum_{j}\uno_{U_j}\leq M$ for some $M>0$, where $\uno_{U_j}$ is the indicator function of $U_j$ so that the sum $\sum_ju_j$ is locally finite. Possibly reducing $\ve_j>0$, as specified later,
we define
\[
f \ceq \begin{cases}
\sum_j u_j(h \star K_{\ve_j}) \mbox{ on }\he^n \setminus Z_h\\
0 \mbox{ on }Z_h,
\end{cases}
\]
where $K_{\ve_j}$ are standard mollification kernels. The function $f$ is clearly smooth on $\he^n \setminus Z_h$, so in order to check that $f$ is continuous on the whole $\he^n$ we  have to check its continuity at points of $Z_h$. 
Up to reducing $\ve_j$, we can assume that
\[
|(h \star K_{\ve_j})-h|\leq d(U_j,Z_h) \mbox{ on }U_j
\]
so that
\[
|f(x)-h(x)|\leq \sum_j u_j(x)|(h \star K_{\ve_j})(x)-h(x)|\leq d(x,Z_h).
\]
This implies the continuity of $f$ since for every $\overline{x}\in Z_h$ one has
\[
\lim_{x \to \overline{x}}|f(x)-f(\overline{x})|=\lim_{x \to \overline{x}}|f(x)|\leq \lim_{x \to \overline{x}}|h(x)|+d(x,\overline{x})=0.
\]
Now we want to prove that $f \in C^1_\he(\he^n,\R^k)$. Since $f$ is by definition smooth on $\he^n \setminus Z_h$ we just need to prove that for every $W \in \hel_1$ and $\overline{x}\in Z_h$
\[
\lim_{x \to \overline{x},x \not\in Z_h}(Wf)(x)=(Wh)(\overline{x}).
\]
For every $x \in \he^n \setminus Z_h$ we get (using $\sum_j Wu_j \equiv 0$)
\begin{align*}
&|(Wf)(x)-(Wh)(x)|\\
=&\left| \sum_j (Wu_j)(x)(h\star K_{\ve_j})(x)+ \sum_j u_j(x)(Wh\star K_{\ve_j})(x)-(Wh)(x) \right|\\
\leq &\left| \sum_j (Wu_j)(x)[(h\star K_{\ve_j})(x)-h(x)]\right|+\left| \sum_j u_j(x)[(Wh\star K_{\ve_j})(x)-(Wh)(x)] \right|\\
 \leq & Cd(x,Z_h)
\end{align*}
and letting $x \to \overline{x}$ we get the continuity of the horizontal derivatives, i.e., property (1). 

Now we prove property (2), i.e., that $|\nabla_\he f|$ is bounded. It is enough to prove that $|\nabla_\he f|$ is bounded on $\he^n \setminus Z_h$. We observe that
\begin{align*}
\nabla_\he f&= \sum_j (h \star K_{\ve_j})\otimes (\nabla_\he u_j)+\sum_j u_j ((\nabla_\he h)\star K_{\ve_j})\\
&=\sum_j(h \star K_{\ve_j}-h) \otimes (\nabla_\he u_j)+\sum_j u_j((\nabla_\he h) \star K_{\ve_j})
\end{align*}
on $\he^n \setminus Z_h$.
The last sum is bounded by $\sup |\nabla_\he h|$. Up to reducing $\ve_j$, we can assume that
\[
|h \star K_{\ve_j}-h|\leq (\sup |\nabla_\he u_j|)^{-1} \quad\text{on }U_j
\]
so that  the second to last sum is bounded by $M$. 

As for property (3), we observe that $\{f=0\}=Z_h$: in fact, by~\cite[Theorem~1.4]{vittone} both $\{f=0\}$ and $Z_h$ are entire intrinsic graphs on $\exp($span$\{X_{k+1},\dots, Y_n,T\})$. Since $Z_h\subset\{f=0\}$, the two sets coincide, hence  $f$ is smooth on $\he^n\setminus Z_h=\he^n\setminus\{f=0\}$.

We  prove  property $(4)$; using~\eqref{eq_1/2} we get
\begin{align*}
\widehat{\nabla}_\he f
&=\sum_j (h \star K_{\ve_j}-h)\otimes (\widehat{\nabla}_\he u_j)+\sum_j u_j((\widehat{\nabla}_\he h) \star K_{\ve_j})\\
&\geq \sum_j(h \star K_{\ve_j}-h)\otimes (\widehat{\nabla}_\he u_j)+\frac12\Id_{k\times k}.
\end{align*}
Given $\eta>0$, possibly reducing $\ve_j$ we can assume that
\[
|h \star K_{\ve_j}-h|\leq \eta (\sup |\widehat{\nabla}_\he u_j|)^{-1}\mbox{ on }U_j
\]
so that $\sup_{\he^n}|\sum_j (h \star K_{\ve_j}-h)\otimes (\widehat{\nabla}_\he u_j)|\leq M \eta$. If $\eta$ is small enough we get
\[
\widehat{\nabla}_\he f \geq  \frac 14 \Id_{k \times k},
\]
which is property (4).

Eventually, property $(5)$ trivially follows upon setting $U=U(p,r)$ and noticing that
\[
S \cap U=\lbrace p \in U : g(p)=0\rbrace
=\lbrace p \in U :  h(p)=0\rbrace
=\lbrace p \in U :  f(p)=0\rbrace.
\]
The proof is accomplished.
\end{proof}

\begin{defi}
Let $S\subset \he^n$. The {\em tangent cone} to $S$ at 0 is the set
\[
T^\he_0 S \ceq \lbrace x\in\he^n:\text{there exist } (x_h)_{h\in \N}\subset S,r_h \to +\infty \text{ and }\delta_{r_h}x_h \to x
\rbrace,
\] 
while the tangent cone to $S$ at $p$ is defined as
\[
T^\he_p S \ceq T^\he_0\tau_{p^{-1}}(S).
\]
\end{defi}

We have the following relationship between tangent cone and tangent group.
\begin{pro}[{\cite[Proposition 3.29]{franchi}}]
\label{teoten}
Let $S \subset \he^n$ be a $C^1_\he$ submanifold of codimension $1 \leq k \leq n$. Then
\[
T^\he_p S=T^g_\he S(p)\qquad\text{for every }p\in S.
\]
\end{pro}

The following Lemma~\ref{lem_approxtan} follows from~\cite[Lemma 2.14 (iii)]{jnv}.

\begin{lem}\label{lem_approxtan}
Let $S \subset \he^n$ be a $C^1_\he$ submanifold of codimension $1 \leq k \leq n$. Then, for every compact set $K\subset S$ and every $\ve>0$ there exists $\delta>0$ such that
\[
\forall\;p\in K,\ \forall\;q\in S\cap U(p,\delta)\qquad d(q,p\cdot T_p^\he  S)<\ve d(p,q).
\]
\end{lem}

\begin{defi}\label{def_split}
Assume $\mathfrak{h}=\mathfrak{v} \oplus \mathfrak{w}$ with $\mathfrak{v}$ and $\mathfrak{w}$ subalgebras. Then we define the subgroups $\mathbb{V}=\exp(\mathfrak{v})$ and $\mathbb{W}=\exp(\mathfrak{w})$. We say that a set $S \subset \he^n$ is an (continuous, respectively smooth) \emph{intrinsic graph} over $\W$ along $\V$ if there is a (continuous, resp. smooth) function $\phi:E \subset \W \to \V$ such that $S=gr_\phi$, where by definition
\[
gr_\phi \ceq \lbrace \xi \cdot \phi(\xi) :\xi \in E \rbrace.
\]
We define the \emph{graph map} $\Phi:E \to \H^n$ associated to $\phi$ as $\Phi(\xi) \ceq \xi \cdot \phi(\xi)$ for every $\xi \in E$.
If, moreover, the subalgebras $\mathfrak{w}$ and $\mathfrak{v}$ are orthogonal, we say that $\he^n=\W \cdot \V$ is an \emph{orthogonal splitting} of $\he^n$ and $S$ is an \emph{orthogonal graph}. 
\end{defi}

\begin{rem}\label{rem_proiezioniplitting}
Given a splitting $\he^n=\W \cdot \V$, we will denote by $\pi_\W:\H^n \to \W$ and $\pi_\V:\H^n \to \V$ the projections uniquely defined by $p=\pi_\W(p)\cdot\pi_\V(p)$ for every $p\in\he^n$.
\end{rem}

Combining Lemma~\ref{fsmooth} with~\cite[Theorem 4.1]{franchi},~\cite[Lemma 2.10]{jnv},~\cite[Theorem 8.7 and Theorem 1.3]{cornimagnani2} and~\cite[Theorem 1.6 and Proposition 2.10]{vittone}, we can state the following result.

\begin{teo}
\label{fond1}
Let $S \subset \he^n$ be a $C^1_\he$ submanifold of codimension $1 \leq k \leq n$. Then $S$ is locally an orthogonal continuous graph, i.e., for each $p \in S$ there exist
\begin{itemize}
\item an orthogonal splitting $\H^n=\W \cdot \V$ where $\V$ is a horizontal $k$-dimensional vector space,
\item relatively open sets $A \subset \W$ and $B \subset \V$ such that $p \in U \ceq A\cdot B$,
\item a continuous function $\phi:A \to B$
\item a function $f=(f_1,...,f_k) \in C^1_\he(U,\R^k)$ with $\nabla_\he f_1 \wedge \cdots \wedge \nabla_\he f_k \neq 0$
\end{itemize}
such that 
\[
S \cap U=\lbrace q \in U :f(q)=0 \rbrace=\lbrace \xi \cdot \phi(\xi) :\xi \in A \rbrace.
\]
Further, if $v_1,\dots,v_k\in\mathfrak v$ form an orthonormal basis of the Lie algebra $\mathfrak v\subset \mathfrak h$ of $\V$ and we put
\[
\Delta(q) \ceq |\det[v_if_j(q)]_{1\leq i,j \leq k}| \neq 0 \mbox{ for }q \in U
\]
then
\begin{equation}\label{eq_formulaareaFSSC}
\mathcal{S}^{Q-k}\res(S \cap U)=C_{n,k} \Phi_\# \left( \left( \frac{|\nabla_\he f_1 \wedge \cdots \wedge \nabla_\he f_k|}{\Delta} \circ \Phi \right) \mathcal{H}^{2n+1-k}_E \res \W \right)
\end{equation}
where $\Phi$ is the graph map associated to $\phi$, $\Phi_\#$ denotes push-forward of measures, $\mathcal{H}^{2n+1-k}_E$ denotes the classical Euclidean $(2n+1-k)$-Hausdorff measure and $C_{n,k}$ is a positive constant only depending on $n,k$ and the distance $d$. 

Moreover, the function $f$ above can be chosen to be defined in the whole $\he^n$, and also in such a way that $f \in C^1_\he(\he^n,\R^k) \cap C^\infty(\he^n \setminus \lbrace x \in \he^n:f(x)=0\rbrace,\R^k)$ and there exists a sequence  of smooth functions  $\phi_h: \W \to \V$ such that
\begin{align*}
& \phi_h \to \phi \mbox{ uniformly in A as }h \to +\infty\\
& gr_{\phi_h}=\{q\in \he^n : f(q)=(\tfrac 1h,0,\dots,0)\}.
\end{align*}
\end{teo}

\begin{rem}\label{rem_cnk}
An explicit expression of the constant $C_{n,k}$ in Theorem \ref{fond1} is given by
\[
C_{n,k}=\left( \sup \lbrace \mathcal{L}^{2n+1-k}(\W \cap U(p,1)): p \in U(0,1) \rbrace  \right)^{-1},
\]
see \cite{magnani} and \cite{CorniMagnani}.
\end{rem}

\subsection{The Rumin complex}\label{subsec_Rumin}
Recall that the exterior algebras of multivectors were defined in Definition~\ref{def_multialgebre}.

We now introduce the Rumin complex.
\begin{defi}
For a given integer $0\leq k\leq 2n+1$ we define the following sets of covectors:
\begin{align*}
\mathcal{I}^k &\ceq  \lbrace \lambda \wedge \theta+\mu \wedge d\theta :\lambda \in \bwl^{k-1}\hel,\mu \in \bwl^{k-2}\hel \rbrace,
\\
\mathcal{J}^k &\ceq \lbrace \lambda \in \bwl^k \hel : \lambda \wedge \theta=\lambda \wedge d\theta=0 \rbrace,
\end{align*}
where we adopted the convention that $\bwl^i \hel \ceq \lbrace 0 \rbrace$ if $i<0$. Then we define the {\em Heisenberg differential $k$-forms} as 
\begin{align*}
\Omega^k_\he& \ceq C^\infty \left( \he^n, \frac{\bwl^k \hel}{\mathcal{I}^k}\right) \hspace{-2 cm} &&\mbox{ if }0 \leq k \leq n,
\\
\Omega^k_\he& \ceq C^\infty ( \he^n,\mathcal{J}^k) \hspace{-2 cm} && \mbox{ if }n+1 \leq k \leq 2n+1.
\end{align*}
\end{defi}

It is worth noticing that, for $k\geq n+1$, the exterior differentiation\footnote{We use the same symbol, $d$, for both the distance in $\he^n$ and for exterior differentiation; no confusion will ever arise.} $d$ satisfies $d(\Omega^k_\he)\subset(\Omega^{k+1}_\he)$.
For $k\leq n-1$ we have $d(C^\infty(\he^n;\mathcal I^k))\subset C^\infty(\he^n;\mathcal I^{k+1})$; in particular, $d$ passes to the quotient  defining an operator $d:\Omega^k_\he\to\Omega^{k+1}_\he$.

\begin{teo}[See {\cite{rumin}}]
\label{seqrum}
There exists a second-order differential operator $D:\Omega^{n}_\he \to \Omega^{n+1}_\he$ such that the sequence
\[
0 \to \mathbb{R} \to \Omega_\he^0 \xrightarrow{d}\Omega_\he^1 \xrightarrow{d}\cdots \xrightarrow{d}\Omega^n_\he \xrightarrow{D}\Omega^{n+1}_\H \xrightarrow{d}\cdots \xrightarrow{d} \Omega^{2n+1}_\he \to 0
\]
is exact. 
\end{teo}

\begin{defi}
We denote by $d_c:\Omega^k_\he\to\Omega^{k+1}_\he$ the operator $d_c \ceq d$ if $k\neq n$, $d_c \ceq D$ if $k=n+1$.
\end{defi}

\begin{rem}\label{rem_defDRumin}
We give an explicit expression for $D$, see e.g. \cite[Section 3.2]{vittone}. First we define the Lefschetz operator $L:\bwl^k \R^{2n}\to \bwl^{k+2}\R^{2n}$ as $L(\lambda) \ceq \lambda \wedge d\theta$. When $k=n-1$ then $L$ is bijective (see \cite[Proposition 1.1]{bryant}). Now we observe that 
\[
\frac{\bwl^n \hel}{\mathcal{I}^n}=\frac{\bwl^n{\hel_1}}{\lbrace \mu \wedge d \theta :\mu \in \bwl^{n-2}\hel_1 \rbrace}.
\] 
Then we define $D$ on smooth sections of $\bwl^n \hel_1$  as 
\[
D(\alpha) \ceq d(\alpha-\theta \wedge L^{-1}((d\alpha)_{\hel_1})),
\]
where $\alpha$ is a smooth section of $\bwl^n \hel_1$ and $(\cdot)_{\hel_1}$ denotes the horizontal part (i.e. given $\lambda \in \bwl^k \hel$ for $1 \leq k \leq 2n$ then $\lambda_{\hel_1} \in \bwl^{k}\hel_1$ is the unique covector such that $\lambda=\lambda_{\hel_1}+\mu \wedge \theta$ for a (unique) $\mu \in \bwl^{k-1}\hel_1$).
Since $D$ passes to the quotient modulo smooth sections of $\lbrace \mu \wedge d \theta :\mu \in \bwl^{n-2}\hel_1 \rbrace$ we get that $D$ is well defined as a linear operator $\Omega_\he^n \to \Omega_\he^{n+1}$.
\end{rem}

\subsection{Currents and integration on \texorpdfstring{$C^1_\he$}{C1H}-regular submanifolds}\label{subsec_correnti}

\begin{defi}
Given an open set $O \subset \he^n$ we define the Heisenberg differential $k$-forms with compact support in $O$ as
\begin{align*}
\D^k_\he(O)& \ceq C^\infty_c \left( O, \frac{\bwl^k \hel}{\mathcal{I}^k}\right) \hspace{-2 cm} &&\mbox{ if }0 \leq k \leq n,
\\
\D^k_\he(O)& \ceq C^\infty_c (O,\mathcal{J}^k) \hspace{-2 cm} &&\mbox{ if }n+1 \leq k \leq 2n+1.
\end{align*}
We denote the dual of $\D^k_\he(O)$ (endowed with the natural topology) by $\D_{\he,k}(O)$. An element of  $\D_{\he,k}(O)$ is called {\em Heisenberg current} of dimension $k$ (or {Heisenberg $k$-current}).
\end{defi}

\begin{defi}
Let $O \subset \he^n$ be an open set and $\Tcurr \in \mathcal{D}_{\he,k}(O)$ with $1 \leq k \leq 2n+1$. The \emph{boundary} $\pp_c\Tcurr\in \mathcal{D}_{\he,k-1}(O) $ of $\Tcurr$ is the Heisenberg $(k-1)$-current defined by duality with $d_c$, i.e., 
\begin{align*}
&\partial_c \Tcurr(\omega) \ceq \Tcurr(d \omega) \mbox{ if }k \neq n+1\\
&\partial_c \Tcurr(\omega) \ceq \Tcurr(D \omega) \mbox{ if }k = n+1
\end{align*}
for every $\omega \in \mathcal{D}_\he^{k-1}(O)$.
\end{defi}

We want to define the current canonically associated to a $C^1_\H$-regular submanifold; this, of course, requires the notion of  orientability.

Recalling Remark~\ref{rem_C1HC1dimensionebassa}, for submanifolds of low dimension $1 \leq k \leq n$ we define orientability  as the classical Euclidean one.
 
Let then $O \subset \he^n$ be an open set and $S\subset O$ be an oriented $C^1_\he$ submanifold of dimension $1 \leq k \leq n$ and with locally finite measure; we claim that for every smooth $(k-1)$-form $\lambda$ and every smooth $(k-2)$-form $\mu$
\begin{equation}\label{eq_integralinulli}
\int_S\lambda\wedge\theta=0\qquad\text{and}\qquad \int_S\mu\wedge d\theta=0.
\end{equation}
The first equality in~\eqref{eq_integralinulli} is clear because $TS\subset H\H^n=\ker \theta$, see again Remark~\ref{rem_C1HC1dimensionebassa}. Concerning the second equality, let us fix a sequence $(S_j)_j$ of $k$-dimensional submanifolds $S_j\subset S$ with $C^1$ boundary and such that $S_j \nearrow S$; then 
\begin{align*}
\int_S\mu\wedge d\theta
&=(-1)^{k-2}\int_S (d(\mu\wedge\theta)-d\mu\wedge\theta)
= (-1)^{k}\int_S d(\mu\wedge\theta)\\
&= (-1)^{k}\lim_{j\to\infty}\int_{S_j} d(\mu\wedge\theta)=(-1)^{k}\lim_{j\to\infty}\int_{\pp S_j} \mu\wedge\theta=0,
\end{align*}
where we used the fact that $T\pp S_j\subset TS\subset \ker \theta$.

We can now state an important consequence of~\eqref{eq_integralinulli}.

\begin{pro}\label{prop_integrazionebassa}
If $O \subset \he^n$ is an open set, $S\subset O$ is an oriented $C^1_\he$-regular submanifold of dimension $1 \leq k \leq n$ whose measure is locally finite in $O$ and $\omega \in \D^k_\he(O)$ is a  Heisenberg $k$-form, then the integral
\[
\int_S \tilde\omega
\]
does not depend on the  representative $\tilde\omega\in C^\infty_c ( O, \bwl^k \hel)$ of the equivalence class $\omega$ in the quotient $\D_\H^k(O)=C^\infty_c ( O, {\bwl^k \hel}/{\mathcal{I}^k})$. 
\end{pro}

Proposition~\ref{prop_integrazionebassa} implies that the following definition is well-posed. 

\begin{defi}
\label{defcorlow}
Let $O \subset \he^n$ be an open set and $S$ be an oriented $C^1_\he$ submanifold of dimension $1 \leq k \leq n$ such that $\s^k\res S$ is locally finite in $O$\footnote{By Remark~\ref{rem_C1HC1dimensionebassa}, the measure $\s^k\res S$ is locally finite in $O$ if and only if the $k$-dimensional Euclidean Hausdorff measure on $S$ is locally finite in $O$.}. If $ \omega \in \D^k_\he(O)$ we define
\[
\int_S \omega \ceq \int_S \tilde\omega
\]
for any representative $\tilde\omega\in C^\infty_c ( O, \bwl^k \hel)$ of the equivalence class $\omega$ in the quotient $\D_\H^k(O)=C^\infty_c ( O, {\bwl^k \hel}/{\mathcal{I}^k})$.

\noindent We denote by $\curr{S} \in \D_{\he,k}(O)$ the  Heisenberg current  associated with $S$ and defined by
\begin{equation}\label{eq_integrazionedimensionebassa}
\curr{S}(\omega) \ceq \int_S \omega\qquad \forall\: \omega \in \D^k_\he(O).
\end{equation}
\end{defi}

Let us now consider the case of submanifolds of low codimension; we use the following definition of orientability, see also~\cite{Canarecci_orientabilita}.

\begin{defi}\label{defi_orientabilecodimbassa}
Let $S$ be a $C^1_\he$ submanifold of codimension $1 \leq k \leq n$. We say that $S$ is orientable if we can define $t^\he_S$ (or, equivalently, $n^\he_S$) continuously on the whole $S$.
\end{defi}

We will use the following lemma to define Heisenberg currents canonically associated with $C^1_\he$-regular oriented submanifold.
\begin{lem}[{\cite[Lemma 3.31]{vittone}}]
\label{eqcurr}
Let $n \geq 1$ and $1 \leq k \leq n$ be integers and let $O\subset\he^n$ be open. Then for every oriented $C^1$  submanifold $S\subset O$ of codimension $k$ which is also  $C_\he^1$ regular and such that $\s^{Q-k}\res S$ is locally finite in $O$, we have 
\[
\int_S \langle t^\he_S|\omega   \rangle d\s^{Q-k}=C_{n,k}\int_S \omega\qquad\forall\:\omega\in\D^{2n+1-k}_\he(O).
\]
where $C_{n,k}$ is as in Remark \ref{rem_cnk}.
\end{lem}

\begin{defi}
Let $O\subset \he^n$ be an open set and $S$ be an oriented $C^1_\he$ submanifold of codimension $1 \leq k \leq n$ such that the measure $\s^{Q-k}\res S$ is locally finite in $O$. For every $\omega\in\D^{2n+1-k}_\he(O)$ we define
\[
\int_S\omega \ceq \frac{1}{C_{n,k}}\int_S \langle t^\he_S|\omega   \rangle d\s^{Q-k},
\]
where  $C_{n,k}$ is as in Remark \ref{rem_cnk}.
The  Heisenberg current $\curr{S} \in \mathcal{D}_{\he,2n+1-k}(O)$ associated with $S$ is 
\[
\curr{S}(\omega) \ceq \int_S\omega
\qquad\forall\:\omega\in\D^{2n+1-k}_\he(O).
\] 
\end{defi}

Lemma \ref{eqcurr} states that, when $S$ is both $C^1$ smooth and $C^1_\he$ regular of codimension $k\leq n$, then  the Heisenberg current associated with $S$ coincide with the usual Euclidean one on $\D^{2n+1-k}_\he$.

\begin{lem}
\label{approx}
Let $O \subset \he^n$ be an open set, $S \subset O$ be a $C^1_\he$ submanifold of codimension $1 \leq k \leq n$ such that $\s^{Q-k}\res S$ is locally finite in $O$. Let $\curr{S} \in \mathcal{D}_{\he,2n+1-k}(O)$ be the associated current. Then for each $p \in S$ there exist an open set $U$ with $p \in U$ and a sequence $(S_h)_{h \in \mathbb{N}}$ of smooth $C^1_\he$ submanifolds of codimension $k$ with locally finite measure on $O$ such that, for every $\omega \in \mathcal{D}^{2n+1-k}_\he(U)$,
\[
\curr{S_h}(\omega) \xrightarrow{h \to +\infty} \curr{S}(\omega).
\]
\end{lem}
\begin{proof}
By Theorem \ref{fond1} for each point $p \in S$ there exists an open set $U$, a function $f=(f_1,...,f_k) \in C^1_\he(\he^n,\R^k)\cap C^\infty(\he^n\setminus\{f=0\},\R^k)$ such that $\nabla_\he f_1 \wedge \cdots \wedge \nabla_\he f_k \neq 0$ on $\he^n$, an orthogonal splitting $\he^n=\W \cdot \V$ and a map $\phi: E \subset \W \to \V$ such that
\[
S \cap U=\lbrace x \in U :f(x)=0 \rbrace=\lbrace \xi \cdot \phi(\xi) :\xi \in E \rbrace
\]
The smooth submanifolds $S_h \ceq \left \lbrace f=(\frac{1}{h},0,...,0) \right\rbrace$ satisfy (possibly restricting the subsets $U$ and $E$)
\[
S_h \cap U=\left\lbrace x \in U :f(x)=\left(\tfrac{1}{h},0,...,0\right) \right\rbrace=\lbrace \xi \cdot \phi_h(\xi) :\xi \in E \rbrace 
\]
for suitable smooth maps $\phi_h:E \to \mathbb{V}$ that converge uniformly to $\phi$ on $E$. For $\omega \in \mathcal{D}^{2n+1-k}_\he(U)$ we have  
\[
\curr{S_h}(\omega)=\frac{1}{C_{n,k}}\int_{S_h}\langle t^\he_{S_h}(q)|\omega(q) \rangle d\s^{Q-k}(q)=\frac{1}{C_{n,k}}\int_{S_h \cap U}\langle t^\he_{S_h}(q)|\omega(q) \rangle d\s^{Q-k}(q).
\]
We define $\Phi_h(\xi) \ceq \xi \cdot \phi_h(\xi)$ and $\Delta_h(p)$ as
\[
\Delta_h(p) \ceq \left|	\det\left[v_i \left(f_j-\frac{\delta_{1j}}{h}\right)(p)\right]_{1 \leq i,j \leq k}\right|
\]
where  $v_i\in\hel_1$ are those determined  in Theorem \ref{fond1}. By the same Theorem we get
\begin{align*}
\curr{S_h}(\omega)&=\frac{1}{C_{n,k}}\int_{S_h \cap U}\langle t^\he_{S_h}(q)|\omega(q) \rangle d\s^{Q-k}(q)=\\
&=\int_{E} \langle t^\he_{S_h}(\Phi_h(\xi))|\omega(\Phi_h(\xi))\rangle \frac{|\nabla_\he \left(f_1-\frac{1}{h}\right)\wedge \nabla_\he f_2\wedge \cdots \wedge \nabla_\he f_k |}{\Delta_h}	(\Phi_h(\xi))  d\xi
\end{align*}
Now we make the following observations:
\begin{itemize}
\item[(a)] Since $\phi_h \xrightarrow{h \to +\infty} \phi$ uniformly on $E$, then $\Phi_h \xrightarrow{h \to +\infty} \Phi$ (where $\Phi(\xi)\ceq \xi \cdot \phi(\xi)$) uniformly on $E$.
\item[(b)] $\nabla_\he \left( f_1-\frac{1}{h} \right)=\nabla_\he f_1$ and $f \in C^1_\he(U,\R^k)$; in particular, possibly restricting $E$ and  $U$,
\[
\left( \nabla_\he \left(f_1-\frac{1}{h}\right) \wedge  \cdots \wedge \nabla_\he f_k \right) \circ \Phi_h \xrightarrow{h \to +\infty}\left( \nabla_\he f_1  \wedge \cdots \wedge \nabla_\he f_k \right) \circ \Phi
\]
uniformly on $E$.
\item[(c)] Also $t^\he_{S_h} \circ \Phi_h \xrightarrow{h \to +\infty}t^\he_{S} \circ \Phi$ uniformly on $E$ because
\begin{align*}
\lim_{h \to +\infty} t^\he_{S_h}(\Phi_h(\xi))&=*\left(\lim_{h \to +\infty} n^\he_{S_h}(\Phi_h(\xi))\right)
\\
&=*\left( \lim_{h \to +\infty}\frac{\nabla_\he \left(f_1-\frac{1}{h}\right) \wedge  \cdots \wedge \nabla_\he f_k }{|\nabla_\he \left(f_1-\frac{1}{h}\right) \wedge  \cdots \wedge \nabla_\he f_k |}(\Phi_h(\xi))\right)\\
&=*\left( \lim_{h \to +\infty} \frac{\nabla_\he f_1 \wedge  \cdots \wedge \nabla_\he f_k }{|\nabla_\he f_1  \wedge  \cdots \wedge \nabla_\he f_k |}(\Phi_h(\xi))\right)\\
&=*\left(\frac{\nabla_\he f_1 \wedge  \cdots \wedge \nabla_\he f_k }{|\nabla_\he f_1  \wedge  \cdots \wedge \nabla_\he f_k |}(\Phi(\xi))\right)\\
&=*(n^\he_S (\Phi(\xi)))\\
&=t^\he_S(\Phi(\xi)).
\end{align*}
\item[(d)] Since $\omega$ is smooth, then $\omega \circ \Phi_h \xrightarrow{h \to +\infty} \omega \circ \Phi$ uniformly on $E$.
\item[(e)] By definition of $\Delta_h$ (given above) and $\Delta$ (as in Theorem~\ref{fond1}), $\Delta_h \circ \Phi \xrightarrow{h \to +\infty} \Delta \circ \Phi$ uniformly on $E$.
\item[(f)] There exists $C>0$ such that $\Delta_h\geq C$ and $\Delta\geq C$.
\end{itemize}

With these observations in mind we get 
\begin{align*}
\lim_{h \to +\infty}\curr{S_h}(\omega)=
& \lim_{h \to +\infty}\int_{E} \langle t^\he_{S_h}(\Phi_h(\xi))|\omega(\Phi_h(\xi))\rangle \frac{|\nabla_\he \left(f_1-\frac{1}{h}\right)\wedge \cdots \wedge \nabla_\he f_k |}{\Delta_h}(\Phi_h(\xi))d\xi 
\\
=&\int_{E} \langle t^\he_{S}(\Phi(\xi))|\omega(\Phi(\xi))\rangle \frac{|\nabla_\he f_1\wedge \cdots \wedge \nabla_\he f_k |}{\Delta}(\Phi(\xi))d\xi
\\
=&\frac{1}{C_{n,k}}\int_{S\cap U}\langle t^\he_{S}(q)|\omega(q) \rangle d\s^{Q-k}(q)\\
=&\curr{S}(\omega),
\end{align*}
as desired.
\end{proof}

\begin{rem}
Lemma~\ref{approx} shows that it is possible to (locally) approximate Heisenberg currents associated with $C^1_\he$ regular submanifolds by (currents associated with) smooth submanifolds. In Lemmata~\ref{approx2} and~\ref{approx2_crit} we will do the analogue for $C^1_\he$ regular submanifolds with boundary.
\end{rem}

\begin{lem}\label{lem_stessovalore}
Let $1 \leq k \leq n$ and $m \geq 0$ be integers, $O \se \H^n$ be an open set. Let $d_1,d_2$ be two left invariant, homogeneous and rotationally invariant distances on $\H^n$ and let $\s^m_1,\s^m_2$ be the associated spherical Hausdorff measures; denote by $C^1_{n,k}$ and $C^2_{n,k}$ the constants provided by Lemma \ref{eqcurr}. Let $S \se O$ be an oriented $k$-codimensional $C^1_\H$-regular submanifold such that $\s^{Q-k}_1 \res S$ and $\s^{Q-k}_2 \res S$ are locally finite on $O$; then,
\[
\frac{1}{C^1_{n,k}}\int_S \langle t^\H_S|\omega \rangle d\s^{Q-k}_1=\frac{1}{C^2_{n,k}} \int_S \langle t^\H_S|\omega \rangle d \s^{Q-k}_2.
\]
for every $\omega \in \D^{2n+1-k}_\H(O)$.
\end{lem}
\begin{proof}
By Lemma \ref{approx} for every $p \in S$ there exist an open set $U$ with $p \in U$ and a sequence $(S_h)_{h \in \N}$ of smooth and $C^1_\H$ $k$-codimensional submanifolds such that, for every $\omega \in \D^{2n+1-k}_\H(U)$ 
\begin{equation}\label{eq_ccc1}
\int_{S_h}\omega \xrightarrow{h \to +\infty}\frac{1}{C^j_{n,k}}\int_S \langle t^\H_S|\omega\rangle d\s^{Q-k}_j.
\end{equation}
for $j=1,2$. By Lemma \ref{eqcurr}, for every $h \in \N$ we have  
\[
\int_{S_h}\omega=\frac{1}{C_{n,k}^1}\int_{S_h} \langle t^\H_{S_h}|\omega \rangle d \s_1^{Q-k}=\frac{1}{C_{n,k}^2}\int_{S_h} \langle t^\H_{S_h}|\omega \rangle d \s_2^{Q-k}.
\]
Letting $h \to +\infty$ and using \eqref{eq_ccc1} we get
\begin{equation}\label{eq_ccc4}
\frac{1}{C_{n,k}^1}\int_{S} \langle t^\H_{S}|\omega \rangle d \s_1^{Q-k}=\frac{1}{C_{n,k}^2}\int_{S} \langle t^\H_{S}|\omega \rangle d \s_2^{Q-k}.
\end{equation} 
We want to extend this result for every $\omega \in \D^{2n+1-k}_\H(O)$. For each $p  \in  S$ we just showed that there exists an open set $U_p\subset O$ such that \eqref{eq_ccc4} holds for every $\omega \in \D^{2n+1-k}_\H(U_p)$.
We extract from the family $(U_p)_{p \in S}$ a countable (or, possibly, finite) sub-family  $(U_i)_{i \in \mathbb{N}}$ such that $S \subset \bigcup_{i \in \mathbb{N}}U_i$.
We fix a partition of the unity, i.e., functions $\zeta_i\in C^\infty_c(U_i), i\in\N,$ such that 
\begin{equation}\label{eq_ccc5}
0\leq\zeta_i\leq 1\qquad\text{and}\qquad\sum_{i \in \mathbb{N} }\zeta_i=1\text{ on }S
\end{equation}
It is not restrictive to assume that the covering $(U_i)_i$ of $S$ is locally finite, so that the sum in~\eqref{eq_ccc5} is well defined.
Let $\omega \in \mathcal{D}^{2n-k+1}_\he(O)$ be fixed. Then
\begin{align*}
\frac{1}{C_{n,k}^1}\int_{S} \langle t^\H_{S}|\omega \rangle d \s_1^{Q-k}&=\frac{1}{C_{n,k}^1}\sum_{i\in \N}\int_{S \cap U_i} \langle t^\H_{S}|\zeta_i\omega \rangle d \s_1^{Q-k}
\\&= \frac{1}{C_{n,k}^2}\sum_{i\in \N}\int_{S \cap U_i} \langle t^\H_{S}|\zeta_i\omega \rangle d \s_2^{Q-k}=\frac{1}{C_{n,k}^2}\int_{S} \langle t^\H_{S}|\omega \rangle d \s_2^{Q-k}.
\end{align*}
\end{proof}
\begin{rem}
Lemma \ref{lem_stessovalore} shows that Heisenberg currents associated with low codimensinal submanifolds do not depend on the choice of the left invariant, homogeneous and rotationally invariant distance $d$ (and, consequently, from the associated spherical Hausdorff measure). 
\end{rem}

\section{\texorpdfstring{$C^1_\he$}{C1H}-regular submanifolds with boundary}

After providing the definition of $C^1_\he$-regular submanifolds with boundary we will give equivalent characterizations first in the case of low and non-critical codimension $k<n$, then in the (easy) case of low dimension (i.e., codimension $k\geq n+1$), and eventually in the case of critical codimension $k=n$. Then we will introduce and discuss the natural way to induce an orientation on the boundary.

\begin{defi}\label{def_supbound}
Let $1 \leq k \leq 2n+1$ and $O \se \H^n$ be an open set. We say that $S \se \H^n$ is a \emph{$k$-codimensional (or $(2n+1-k)$-dimensional)  $C^1_\he$-regular submanifold with boundary in $O$} if the following conditions hold:
\begin{enumerate}[label=(\arabic*)]
\item $O \cap S$ is a non empty $k$-codimensional $C^1_\he$-regular submanifold;
\item $O \cap \pp S$ is a non empty $k+1$-codimensional $C^1_\he$-regular submanifold;
\item for all $p \in O \cap \pp S$ there exist a neighbourhood $U$ of $p$ and a  $k$-codimensional $C^1_\H$-regular submanifold $S' \se U$ such that $U \cap \overline{S} \se S'$ and, for every $r>0$,  $U(p,r) \cap (S'\setminus \overline{S})\neq \emptyset$.
\end{enumerate}
We will omit reference to the open set $O$ in case the latter is clear from the context.
\end{defi}

\begin{teo}\label{teo_equivsupbound}
Let $1 \leq k \leq n-1$, $O \se \he^n$ be an open set and $S \se O$ such that $O \cap \pp S \neq \emptyset$. Then the following statements are equivalent:
\begin{enumerate}
\item[(i)] $S$ is a $k$-codimensional $C^1_\H$-regular submanifold with boundary in $O$;
\item[(ii)] the statement $(1)$ from Definition \ref{def_supbound} holds and for all $p \in O \cap \pp S$ there exist a neighbourhood $U$ of $p$ and functions $f_1,\dots,f_{k+1}\in C^1_\H(U)$ such that
\begin{align*}
& \text{$\nabla_\H f_1,\dots,\nabla_\H f_{k+1}$ are linearly independent in $U$,}\\
& U \cap S=\lbrace q \in U : f_1(q)=\cdots=f_k(q)=0,f_{k+1}(q)>0 \rbrace,\\
& f_{k+1}(p)=0.
\end{align*}
\end{enumerate}
\end{teo}

\begin{rem}\label{rem_bound}
If $(ii)$ of Theorem \ref{teo_equivsupbound} holds, then
\[
U \cap \overline{S}=\lbrace q \in U : f_1(q)=\cdots=f_k(q)=0,f_{k+1}(q) \geq 0 \rbrace.
\]
In fact, the left hand side in the previous equality is clearly contained in the right hand one; the opposite inclusion can be easily proved taking into account that the map $(f_1,\dots,f_{k+1}):U\to\R^k$ is open (see for instance \cite[Lemma 2.10]{jnv}). In particular,
\[
U \cap \pp S=\lbrace q \in U : f_1(q)=\cdots=f_k(q)=0,f_{k+1}(q) = 0 \rbrace
\]
and $O\cap \partial S$ is a $C^1_\H$-regular submanifold of codimension $k+1$.
\end{rem}

\begin{rem}\label{rem_wlogCinfty}
By a standard mollification procedure (see for instance in \cite[Proposition~2.10]{vittone}), the functions $f_1,\dots,f_{k+1}$ in $(ii)$ of Theorem~\ref{teo_equivsupbound} can be chosen to be of class $C^\infty$ on the open set $U\setminus\{q\in U : f_1(q)=\dots=f_k(q)=0\}$. 
\end{rem}

\begin{proof}[Proof of Theorem \ref{teo_equivsupbound}]
Let us prove $(ii)\Rightarrow (i)$: since properties (1) and (2) in Definition \ref{def_supbound} hold, respectively, by assumption and by Remark~\ref{rem_bound}, we have only to prove property (3).  Fix $p \in O \cap \pp S$. Let $U,f_1,\dots,f_{k+1}$ be as in $(ii)$. We define
\[
S' \ceq \lbrace q \in U : f_1(q)=\cdots=f_{k}(q)=0 \rbrace.
\]
It is clear that $S'$ is a $k$-codimensional $C^1_\H$-regular submanifold and $U \cap \overline{S}\se S'$. We must prove that, for every $r>0$, $U(p,r) \cap (S' \setminus \overline{S}) \neq \emptyset$. The map $(f_1,\dots,f_{k+1}):U \to \R^k$ is open, so for every $r>0$ there exists a $\ve>0$ such that
\[
B(0,\ve)=\lbrace a \in \R^{k+1}:|a|_{\R^{k+1}}<\ve \rbrace \se (f_1,\dots,f_{k+1})(U(p,r)),
\]
since, by Remark \ref{rem_bound}, $f_1(p)=\dots=f_{k+1}(p)=0$.
Therefore, there exists $q \in U(p,r)$ such that $(f_1,\dots,f_{k+1})(q)=(0,\dots,0,-\tfrac \ve 2)$, i.e., $q \in U(p,r) \cap (S' \setminus \overline{S})$, that is $U(p,r) \cap (S' \setminus \overline{S}) \neq \emptyset$.

We now prove $(i) \Rightarrow (ii)$. Fix $p \in O \cap \pp S$ and let $S'$ be as in $(3)$ of Definition~\ref{def_supbound}; then, there exist a neighbourhood $U$ of $p$ and functions $f_1,\dots,f_k \in C^1_\H(U)$ such that, setting
\[
S'=\lbrace q \in U : f_1(q)=\cdots=f_k(q)=0 \rbrace,
\]
we have
\begin{itemize}
\item $\nabla_\H f_1,\dots,\nabla_\H f_{k}$ are linearly independent on $U$,
\item $U \cap \overline{S} \se S'$,
\item for every $r>0$, $U(p,r) \cap (S' \setminus \overline{S}) \neq \emptyset$. 
\end{itemize}
By property $(2)$, possibly reducing $U$, there exist $g_1,\dots, g_{k+1}\in C^1_\H(U)$ such that
\begin{itemize}
\item $\nabla_\H g_1,\dots,\nabla_\H g_{k+1}$ are linearly independent on $U$,
\item $U \cap \pp S=\lbrace q \in U : g_1(q)=\dots=g_{k+1}(q)=0 \rbrace$.
\end{itemize}
Let $\bar\jmath\in \lbrace 1,...,k+1 \rbrace$ be such that $\nabla_\H f_1(p),\dots,\nabla_\H f_k(p), \nabla_\H g_{\bar\jmath}(p)$ are linearly independent. Setting $f_{k+1}=g_{\bar\jmath}$ and, possibly, reducing $U$, $\nabla_\H f_1,\dots, \nabla_\H f_{k+1}$ are linearly independent on $U$. We claim that
\begin{align}
U \cap \pp  S &=\lbrace q \in U : f_1(q)=\cdots=f_{k+1}(q)=0 \rbrace.\label{eq_claim}
\intertext{The inclusion}
U \cap \pp S &\se\lbrace q \in U : f_1(q)=\cdots=f_{k+1}(q)=0 \rbrace\nonumber\\
&=S' \cap \lbrace q \in U : f_{k+1}(q)=0\rbrace\nonumber
\end{align}
holds because  $U \cap \pp S \se U \cap \overline{S} \se U \cap S'$ and, if $q \in U \cap \pp S$, then $f_{k+1}(q)=g_{\bar j}(q)=0$.
For the opposite inclusion 
\[
U \cap \pp S \supseteq \lbrace q \in U : f_1(q)=\cdots=f_{k+1}(q)=0 \rbrace
\]
The set $\Delta \ceq \lbrace q \in U : f_1(q)=\cdots=f_{k+1}(q)=0 \rbrace$ is a $C^1_\H$ submanifold of codimension $k+1$. Notice that $p\in \Delta\cap \partial S$ and
\[
T_p^\H \Delta=T_p^\H \pp S=T_p^\H S' \cap (\nabla_\H g_{\bar \jmath}(p))^\perp.
\]
Possibly restricting $U$, by Theorem~\ref{fond1} there exist a splitting $\H^n=\W \cdot \V$ (with $\V$ horizontal and of dimension $k+1$), open subsets $A\subset\W, B\subset\V$ and continuous maps $\phi,\psi:A \to B$ such that $U=A\cdot B$ and
\[
\quad gr_\psi=U \cap \pp S,\qquad
gr_\phi=U \cap\Delta.
\]
The inclusion $gr_\psi \se gr_\phi$ proved above implies that $\psi=\phi$ and the claim~\eqref{eq_claim} follows. Moreover, possibly reducing $A$ and  $U$, by the Implicit Function Theorem \cite[Lemma 2.10]{jnv}, there exist a real number $t_0>0$ and a continuous map $\vp:A \times(-t_0,t_0) \to \V$ such that 
\begin{itemize}
\item the graph map $\Phi(w,t) \ceq w \cdot \vp(w,t)$ is an homeomorphism between $A \times (-t_0,t_0)$ and $\lbrace q \in S' \st f_{k+1}(q) \in (-t_0,t_0)\rbrace$,
\item $f_{k+1}(w \cdot \vp(w,t))=t$ for all $w \in A$ and $t \in (-t_0,t_0)$.
\end{itemize}
Without loss of generality we can suppose that $A$ is a connected open subset of $\W$. Moreover, possibly reducing $U,A$ and $t_0$, we can suppose that
\begin{itemize}
\item $U \cap \pp S=\Phi(A \times \lbrace 0 \rbrace)$,
\item $U \cap S'=\Phi(A \times (t_0,t_0))$.
\end{itemize}
We observe that $U \cap (S' \setminus \pp S)$ has two connected components  $S'_+$, $S'_-$ defined by
\begin{align*}
&S'_+ \ceq \lbrace q \in U : f_1(q)=\cdots=f_k(q)=0,f_{k+1}(q)>0 \rbrace\\
&S'_- \ceq \lbrace q \in U : f_1(q)=\cdots=f_k(q)=0,f_{k+1}(q)<0 \rbrace.
\end{align*}
By definition $U \cap S$ is a $C^1_\H$-regular submanifold: it follows that it is locally connected, relatively open and relatively closed in  $U \cap (S' \setminus \pp S)=S'_- \sqcup S'_+$. This implies that there are only three possibilities
\begin{enumerate}
\item[(a)] either $U \cap S=S'_+$,
\item[(b)] or $U \cap S=S'_-$,
\item[(c)] or $U \cap S=S'_+ \cup S'_-$.
\end{enumerate} 
If (a) holds, then the proof is accomplished.  If (b) holds, it is enough to replace $f_{k+1}$ with $-f_{k+1}$. Case (c) would lead to  $U \cap (S' \setminus \overline{S})=\emptyset$, a contradiction. This concludes the proof.
\end{proof}

The following property of low-dimensional $C^1_\H$ submanifolds with boundary is an easy consequence of~\cite[Theorem 3.5]{franchi}, see Remark~\ref{rem_C1HC1dimensionebassa}.

\begin{lem}\label{lem_lowdimsubwbound}
Let $1 \leq k \leq n$, $O \se \he^n$ be an open set and let $S$ be a $k$-dimensional $C^1_\H$-regular submanifold with boundary in $O$. Then, $O \cap \overline{S}$ is a Euclidean $k$-dimensional $C^1$-regular submanifold with boundary such that $T \overline{S} \se H\H^n$ and $T \pp S \se H\H^n$.
\end{lem}

In Lemma~\ref{lem_lowdimsubwbound} and in the following, when we say that $S$ is both a Euclidean $C^1$ and $C^1_\H$ submanifold with boundary in an open set $O$ we mean that $O \cap \overline{S}$ is both a $C^1$-regular Euclidean submanifold with boundary (in the classical sense) and $S$ is a $C^1_\H$-regular submanifold with boundary in $O$ as in Definition \ref{def_supbound}.

The following Lemma shows a property of critical dimensional $C^1_\H$ submanifolds with boundary, which will be crucial in the definition of the induced orientation on the boundary in Subsection \ref{sub_boundor}.

\begin{lem}\label{lem_2cc}
Let $O \se \H^n$ be an open set and let $S$ be a $(n+1)$-dimensional $C^1_\H$ submanifold with boundary in $O$; denote by $S'$ a local extension of $S$ as in Definition \ref{def_supbound}. Then for every $p \in O \cap \pp S$ there exists a fundamental system of neighborhoods $\mathcal{N}$ such that, for every neighbourhood $N \in \mathcal{N}$, $(N \cap S') \setminus \pp S$ has two connected components and one of the two is $N \cap S$.
\end{lem}
\begin{proof}
Fix $p \in O \cap \pp S$. By Definition \ref{def_supbound} there exist a neighbourhood $U$ of $p$ and a $(n+1)$-dimensional $C^1_\H$ submanifold $S' \se U$ such that $U \cap \overline{S}\se S'$. By Theorem \ref{fond1} we can locally write $S'$ as the intrinsic graph of an  continuous map $\phi:\W \to \V$, where $\W:=T_p^\H S'$ and $\H^n=\W\cdot\V$ is an orthogonal splitting. We also fix orthonormal bases $w_1,\dots,w_n,T$ of the Lie (sub)algebra of $\W$ and $v_1,\dots,v_n$ of the Lie (sub)algebra of $\V$, respectively; observe that, since $\pp S$ has an horizontal $n$-dimensional tangent space, $T^\H_p\pp S=\spa(w_1(p),\dots, w_n(p))$.
Let $\pi_\W:\H^n \to \W$ be the projection defined in Remark~\ref{rem_proiezioniplitting}; we observe that, possibly restricting $U$, the set $\Gamma:=\pi_\W(U \cap \pp S)$ is an $n$-dimensional $C^1$ submanifold in $\W$. This is due to the fact that, up to restricting $U$, $U \cap \pp S$ and $\pi_\W$ are transverse, i.e., for every $q \in U \cap \pp S$ 
\[
T_q \pp S \cap  \ker d\pi_\W=T_q\pp S \cap \spa (v_1(q),\dots, v_n(q))=\{0\};
\]
in fact, the previous equality is true for $q=p$ and, by continuity, it remains true for $q$ close enough to $p$.
Therefore, for every $w \in \Gamma$ there exists a fundamental system of neighborhoods $\mathcal{M}$ for $\pi_\W(p)$ such that, for any   $M \in \mathcal{M}$, $M \setminus \Gamma$ has two connected components. Now $U \cap S$ is a $C^1_\H$-regular submanifold, hence it is relatively open in $U \cap (S' \setminus \pp S)$. This, together with the fact that, up to restricting $U$, $\pi_\W$ is an homeomorphism from $U\cap S'$ onto its image, implies that $\pi_\W(U \cap S)$ has to be either one of the two connected components of $M \setminus \Gamma$ or the whole $M \setminus \Gamma$. Thanks to property $(3)$ in Definition \ref{def_supbound}, one has $\pi_\W(U \cap  S)\neq M \setminus \Gamma$, so that $\pi_\W(U \cap  S)$ is exactly one of the two connected components of $M \setminus \Gamma$. Then the fundamental system of neighbourhoods $\mathcal{N}$ defined by $\mathcal{N}\ceq U\cap \pi_\W^{-1}(\mathcal{M})$ satisfies the requested property.
\end{proof}

\begin{rem}\label{rem_soprasotto}
Under the assumption and notation of Lemma~\ref{lem_2cc}, we observe that the $n$-dimensional $C^1$ submanifold $\Gamma=\pi_\W(U\cap\pp S)$ can be written, locally around $\pi_\W(p)$, as an $n$-dimensional graph in direction $T$. More precisely, let us identify $\W\equiv\R^{n+1}$ by
\[
\R^{n+1}\ni(s_1,\dots,s_n,t)\longleftrightarrow \exp(s_1w_1+\dots+s_nw_n+tT)\in\W;
\] 
then there exists $\gamma\in C^1(\R^n)$ and a relatively open set $A\subset \W$ containing $\pi_\W(p)$ such that
\[
A\cap\Gamma=\{(s_1,\dots,s_n, t)\in A\subset \W\equiv\R^{n+1}:t=\gamma(s_1,\dots,s_n)\}\,.
\]
In particular, we have that either
\begin{equation}\label{eq_soprailbordo}
A\cap\pi_\W(U\cap S)=\{(s_1,\dots,s_n, t)\in A\subset \W\equiv\R^{n+1}:t>\gamma(s_1,\dots,s_n)\}
\end{equation}
or
\begin{equation}\label{eq_sottoilbordo}
A\cap\pi_\W(U\cap S)=\{(s_1,\dots,s_n, t)\in A\subset \W\equiv\R^{n+1}:t<\gamma(s_1,\dots,s_n)\}.
\end{equation}
We will informally say that {\em $S$ lies above its boundary around $p$} if~\eqref{eq_soprailbordo} holds, while we say that {\em $S$ lies below its boundary around $p$} if~\eqref{eq_sottoilbordo} holds. 
It is a boring task to check that, for every $p\in \pp S$, there exists a neighbourhood $V$ of $p$ with the following property: if $S$ lies above (resp., below) its boundary around $p$, then $S$ lies above (resp., below) its boundary also around every $q\in V\cap\pp S$. 
\end{rem}

Finally, the following lemma contains a technical property of $C^1_\H$-regular submanifolds with boundary that holds in the case of critical codimension $k=n$. Such a property will be useful later in the proof of Lemma \ref{approx2_crit}.

\begin{lem}
\label{ghat}
Let $S$ be a $(n+1)$-dimensional $C^1_\he$ submanifold with boundary. Define $\pi:\he^n\to\R^{2n}$ by $\pi(x,y,t) \ceq (x,y)$. Then, for every $p \in \pp S$ there exist an open set $U$ with $p \in U$, a $C^1_\H$-regular submanifold $S'\subset U$ as in (3) of Definition~\ref{def_supbound}, a locally defining function $f\in C^1_\H(U,\R^n)$ for $S'$, a function $\hat{g}:\R^{2n} \to \R^{n}$ and $\delta_1>0$ such that, up to an isometry of $\H^n$,
\begin{enumerate}
\item[(1)] $\hat{g} \in C^1(\R^{2n},\R^n)$,
\item[(2)] $|\nabla \hat{g}|$ is bounded on $\R^{2n}$,
\item[(3)] $(\nabla \hat{g})\circ \pi=\nabla_\he f$ on $\Sigma=U \cap \pp S$,
\item[(4)] $\col[\pp_{x_1}\hat{g}|\cdots|\pp_{x_n}\hat{g}](z)\geq \delta_1 \Id_{n \times n}$ for every $z \in\R^{2n}$ in the sense of quadratic forms,
\item[(5)] $\pi(\Sigma) \se \lbrace \hat{g}=0 \rbrace$,
\end{enumerate}
where in (3) we  identify horizontal vectors with elements in $\R^{2n}$.
\end{lem}
\begin{proof}
Fix $p \in \pp S$ and let $U,S'$ be an open neighbourhood of $p$ and an $(n+1)$-dimensional $C^1_\he$ submanifold as in (3) of Definition \ref{def_supbound}. By Lemma \ref{fsmooth}, possibly restricting $U$ there exist a function $f:\he^n \to \R^n$ and $\delta_0>0$ such that, up to an isometry of $\he^n$,
\begin{itemize}
\item $f \in C^1_\he(\he^n,\R^n)$,
\item $|\nabla_\he f|$ is bounded on $\he^n$,
\item $f \in C^\infty(\he^n \setminus \lbrace f=0 \rbrace,\R^n)$,
\item $\widehat \nabla_\he f(q)\geq  \delta_0 \Id_{n \times n}$ for every $q \in \he^n$ in the sense of quadratic forms,
\item $S' \cap U=\lbrace q \in U : f(q)=0\rbrace$.
\end{itemize}
By Remark~\ref{rem_C1HC1dimensionebassa}, $\partial S$ is a Euclidean $C^1$ submanifold whose  tangent space is contained in the horizontal distribution. In particular, $\pp S$ is transversal to the vector field $T$ and, up to reducing $U$,  $\pi|_{U \cap \pp S}$ is a $C^1$ diffeomorphism between $\Sigma=U \cap \pp S$ and its image. We consider, for $1 \leq j \leq n$, the map 
\[
(\nabla_\he f_j) \circ \pi^{-1}: \pi(\Sigma) \to \R^{2n},
\]
where $\nabla_\he f_j(p)\in H_p\he^n$ is identified with $(X_1f_j(p),...,Y_n f_j(p))\in\R^{2n}$. Clearly, $\nabla_\he f_j \circ \pi^{-1}$ is continuous and bounded on $\pi(\Sigma)$. Up to reducing  $U$, we can suppose that  $\nabla_\he f_j \circ \pi^{-1}$ is defined on $\overline{\pi(\Sigma)}$. We use the classical Whitney Extension Theorem to extend the null function on $\overline{\pi(\Sigma)}$ to a $C^1$ map with gradient $\nabla_\he f_j \circ \pi^{-1}$ on $\overline{\pi(\Sigma)}$. In order to use Whitney Extension Theorem we  have to check that for   every $\ve>0$ there exists $\delta>0$ such that
\[
|\langle \nabla_\he f_j \circ \pi^{-1}(z),z-z'\rangle_{\R^{2n}}|\leq \ve |z-z'|_{\R^{2n}}\qquad \text{for every  }z,z' \in \overline{\pi(\Sigma)}\text{ with }|z-z'|_{\R^{2n}}<\delta.
\]
Let $z,z' \in \overline{\pi(\Sigma)}$ be fixed. Since, up to restricting $U$, $\pi^{-1}$ is bijective on $\overline{\pi(\Sigma)}$, there exist unique $q,q' \in \overline\Sigma$ such that $\pi(q)=z$ and $\pi(q')=z'$. We have to prove that for every $\ve>0$ there exist $\delta>0$ such that, if $q,q'\in \overline\Sigma$ and $|\pi(q)-\pi(q')|_{\R^{2n}}<\delta$ then
\begin{equation}\label{ghat1}
|\langle \nabla_\he f_j (q),\pi(q)-\pi(q')\rangle_{\R^{2n}}|\leq \ve |\pi(q)-\pi(q')|_{\R^{2n}}.
\end{equation}
This is equivalent to
\[
|\langle \nabla_\he f_j (q),\pi(q^{-1} \cdot q)\rangle_{\R^{2n}}|\leq \ve |\pi(q^{-1} \cdot q)|_{\R^{2n}}.
\]
Let  $r \in q \cdot T_q^\he  S'$ such that $d(q',q \cdot T_q^\he  S')=d(q',r)$. Then
\begin{align*}
|\langle \nabla_\he f_j (q),\pi(q^{-1} \cdot q')\rangle_{\R^{2n}}|&\leq \underbrace{|\langle \nabla_\he f_j (q),\pi(q^{-1} \cdot r)\rangle_{\R^{2n}}|}_{=0}+|\langle \nabla_\he f_j (q),\pi(r^{-1} \cdot q')\rangle_{\R^{2n}}|
\\
&\leq  \underbrace{| \nabla_\he f_j (q)|_{\R^{2n}}}_{\leq C}\underbrace{|\pi(r^{-1} \cdot q')|_{\R^{2n}} }_{\leq Cd(r,q')}\leq Cd(r,q')
\end{align*}
By Lemma \ref{lem_approxtan} there exist $\delta_2>0$ such that, if $d(q,q')<\delta_2$, then
\[
d(r,q') \leq \ve d(q,q').
\]
By Lemma \ref{lem_distlowdim} there exist a $\delta_1>0$ and a constant $C'>0$ such that, if $d(p,q')<\delta_1$ and $d(p,q)<\delta_1$, then
\[
d(q,q')\leq C'|\pi(q)-\pi(q)'|_{\R^{2n}}.
\]
The claim~\eqref{ghat1} follows and we can use the (classical Euclidean) Whitney Extension Theorem (see e.g \cite[Theorem 6.10]{evansgariepy}) to obtain a function $g_j \in C^1(\R^{2n},\R)$ such that $g_j|_{\overline{\pi (\Sigma)}}=0$ and $(\nabla g_j)\circ\pi=\nabla_\he f_j $ on $\Sigma$. We define the function $g: \R^{2n} \to \R^n$ as $g=(g_1,...,g_n)$. 
The function $g$ satisfies statements (1), (3) and (5) (with $g$ in place of $\hat g$), while it satisfies statements (2) and (4) only in a neighbourhood of $\overline{\pi(\Sigma)}$. To conclude, one can use a  strategy as in the proof of Lemma \ref{fsmooth} to get a function $\hat{g}$ with the desired properties.
\end{proof}

\subsection{Boundary orientation}\label{sub_boundor} In this subsection we show that, given an orientation on a $C^1_\H$-regular submanifold with boundary, there exists a natural way to define an induced orientation on the boundary which will be consistent with Theorem \ref{thm_intro}, in analogy with the classical Euclidean case. In the low dimensional case, because of Lemma \ref{lem_lowdimsubwbound}, we can define the induced orientation on the boundary as usual, so we focus first on the low codimensional case and then on the critical dimensional case. 

\begin{rem}
Let $1 \leq k \leq n$, $O \se \H^n$ be an open set and $S \se \H^n$ a $k$-codimensional $C^1_\H$ submanifold with boundary in $O$. Let $p \in O \cap \pp S$ and $S'$ be the local extension of $S$ as in Definition \ref{def_supbound}. Then the tangent cone $T_p^\H S'$ does not depend on the choice of the local extension $S'$ and we denote it by $T_p^\H S$.
\end{rem}

\begin{lem}\label{lem_fk+1}
Let $1 	\leq k\leq n-1$, $O \se \H^n$ be an open set and $S \se \H^n$ a $k$-codimensional $C^1_\H$ submanifold with boundary in $O$. Let $p \in O \cap \pp S$ and assume there exists an open set $U$ with $p \in U$ and two families of locally defining functions $f_1,\dots,f_{k+1} \in C^1_\H(U)$ and $g_1,\dots,g_{k+1} \in C^1_\H(U)$ for $S$, as in Theorem \ref{teo_equivsupbound}. Then there exists a positive constant $C>0$ such that
\[
(d_\H f_{k+1})_p=C( d_\H g_{k+1})_p \quad \text{ on }T_p^\H S.
\]
\end{lem}

\begin{proof}
By Theorem \ref{teo_equivsupbound} and Remark \ref{rem_bound} we have
\begin{align*}
& \text{$\nabla_\H f_1,\dots,\nabla_\H f_{k+1}$ are linearly independent in $U$,}\\
& U \cap S=\lbrace q \in U : f_1(q)=\cdots=f_k(q)=0,f_{k+1}(q)>0 \rbrace,\\
& U \cap \pp S=\lbrace q \in U : f_1(q)=\cdots=f_{k+1}(q)=0 \rbrace,\\
& \text{$\nabla_\H g_1,\dots,\nabla_\H g_{k+1}$ are linearly independent in $U$,}\\
& U \cap S=\lbrace q \in U : g_1(q)=\cdots=g_k(q)=0,g_{k+1}(q)>0 \rbrace,\\
& U \cap \pp S=\lbrace q \in U : g_1(q)=\cdots=g_{k+1}(q)=0 \rbrace.
\end{align*}
We claim that
\[
\lim_{\lambda \to 0}\delta_{1/\lambda}(\tau_{p^{-1}}S)=T_p^\H S \cap \lbrace q \in \H^n : (d_\H f_{k+1})_p(q)>0\rbrace.
\]
in the sense of local Hausdorff convergence of sets. For the sake of brevity, let us assume $p=0$. For $\lambda>0$ we define for $i=1,\dots,k+1$ the maps $(f_i)_\lambda \ceq \tfrac 1 \lambda (f_i \circ \delta_{\lambda})$ and the $C^1_\H$ submanifolds $S_\lambda \ceq \delta_{1/\lambda}( U \cap S)$, $(\pp S)_\lambda \ceq \delta_{1/\lambda}(U \cap \pp S)$, $S'_\lambda \ceq \delta_{1/\lambda}(S')$ where $S'\ceq \lbrace q \in U: f_1(q)=\dots f_k(q)=0\rbrace$. We observe that
\begin{align*}
S_\lambda&=\lbrace q \in \delta_{1/\lambda}(U) : (f_1)_\lambda(q)=\dots=(f_k)_\lambda(q)=0,(f_{k+1})_\lambda(q)>0\rbrace,\\
(\pp S)_\lambda&=\lbrace q \in \delta_{1/\lambda}(U) : (f_1)_\lambda(q)=\dots=(f_{k+1})_\lambda(q)=0\rbrace,\\
S'_\lambda&=\lbrace q \in \delta_{1/\lambda}(U) : (f_1)_\lambda(q)=\dots=(f_k)_\lambda(q)=0\rbrace
\end{align*}
and in particular that
\begin{align*}
(\pp S)_\lambda&=S_\lambda' \cap \lbrace q \in \delta_{1/\lambda}(U) : (f_{k+1})_\lambda(q)=0 \rbrace,\\
S_\lambda&=S'_\lambda \cap \lbrace q \in \delta_{1/\lambda}(U) : (f_{k+1})_\lambda(q)>0 \rbrace.
\end{align*}
Letting $\lambda \to 0$, by \cite[Lemma 2.14 (ii) and (iii)]{jnv} we obtain 
\begin{align*}
T_0^\H \pp S&=T_0^\H S \cap \lbrace q \in \H^n : (d_\H f_{k+1})_0(q)=0\rbrace,\\
\lim_{\lambda \to 0}\delta_{1/\lambda}(S)&=T_0^\H S \cap   \lbrace q \in \H^n : (d_\H f_{k+1})_0(q)>0\rbrace 
\end{align*}
in the sense of local Hausdorff convergence of sets, proving the claim. In the same way we obtain
\begin{align*}
T_0^\H \pp S&=T_0^\H S \cap \lbrace q \in \H^n : (d_\H g_{k+1})_0(q)=0\rbrace,\\
\lim_{\lambda \to 0}\delta_{1/\lambda}(S)&=T_0^\H S \cap   \lbrace q \in \H^n : (d_\H g_{k+1})_0(q)>0\rbrace 
\end{align*}
in the sense of local Hausdorff convergence of sets. By the uniqueness of the blow-up we obtain  
\begin{align*}
T_0^\H S \cap   \lbrace q \in \H^n : (d_\H f_{k+1})_0(q)>0\rbrace&=T_0^\H S \cap   \lbrace q \in \H^n : (d_\H g_{k+1})_0(q)>0\rbrace,\\
T_0^\H S \cap \lbrace q \in \H^n : (d_\H f_{k+1})_0(q)=0\rbrace&=T_0^\H S \cap \lbrace q \in \H^n : (d_\H g_{k+1})_0(q)=0\rbrace.
\end{align*}
This, together with the fact that $(d_\H f_{k+1})_0$ and $(d_\H g_{k+1})_0$ are  linear maps (on $T_0^\H S $) with the same kernel, implies that there exists a positive constant $C>0$ such that 
\[
d_\H f_{k+1}=C d_\H g_{k+1} \quad \text{ on }T_p^\H S,
\]
as desired.
\end{proof}

\begin{defi}\label{def_outpoint}
Let $1 	\leq k\leq n-1$, $O \se \H^n$ be an open set and $S \se \H^n$ a $k$-codimensional $C^1_\H$ submanifold with boundary in $O$. Fix $p \in O \cap \pp S$ and let $U$ and $f_1,...,f_{k+1} \in C^1_\H(U)$ be as in Theorem \ref{teo_equivsupbound}. We say that a vector $v \in T^\H_p S$ is \emph{outward pointing} if $(d_\H f_{k+1})_p(v)<0$. 
\end{defi}

Definition \ref{def_outpoint} is well posed and does not depend on the choice of the family of local defining functions $f_1,\dots,f_{k+1}$ because of Lemma \ref{lem_fk+1}.

\begin{defi}
Let $1 	\leq k\leq n-1$, $O \se \H^n$ be an open set and $S \se \H^n$ a $k$-codimensional $C^1_\H$ submanifold with boundary in $O$. We say that $V: O \cap \pp S \to T^\H_p S$ is an \emph{outward pointing vector field} if, for any $p \in O \cap \pp S$, $V(p)$ is an outward pointing vector.
\end{defi}

\begin{lem}\label{lem_vfoutpt}
Let $1 	\leq k\leq n-1$, $O \se \H^n$ be an open set and $S \se \H^n$ a $k$-codimensional $C^1_\H$ submanifold with boundary in $O$. Then there exists a unique continuous unit outward pointing vector field $\nu_{\pp S}$ along $O \cap \pp S$ such that $\nu_{\pp S}(p) \perp T^\H_p \pp S$ for every $p \in O \cap \pp S$.
\end{lem}
\begin{proof}
Fix $p \in O \cap \pp S$. By Theorem \ref{teo_equivsupbound} and Remark \ref{rem_bound} there exist a neighbourhood $U$ of $p$
and functions $f_1,...,f_{k+1} \in C^1_\H(U)$ such that
\begin{align*}
& \text{$\nabla_\H f_1,\dots,\nabla_\H f_{k+1}$ are linearly independent in $U$,}\\
& U \cap S=\lbrace q \in U : f_1(q)=\cdots=f_k(q)=0,f_{k+1}(q)>0 \rbrace,\\
& U \cap \pp S=\lbrace q \in U : f_1(q)=\cdots=f_{k+1}(q)=0 \rbrace.
\end{align*}
We define the continuous vector field on $U \cap \pp S$
\[
\nu_{\pp S} (q)\ceq -\frac{\pi_{T_p^\H { S}}\nabla_\H f_{k+1}(q)}{|\pi_{T_p^\H { S}}\nabla_\H f_{k+1}(q)|},\qquad q\in U \cap \pp S
\]
where $\pi_{T_p^\H S}$ denotes the orthogonal projection on $T_p^\H S$. By Lemma \ref{lem_fk+1} it is clear that $\nu_{\pp S}$ does not depend on the choice of the set of the local defining functions $f_1,\dots,f_{k+1}$ so that we can continuously extend $\nu_{\pp S}$ on the whole $O \cap \pp S$.
\end{proof}

\begin{defi}
Under the assumptions and notation of Lemma \ref{lem_vfoutpt}, we call $\nu_{\pp S}$  the \emph{unit outward normal} to $\pp S$.
\end{defi}

\begin{defi}
Let $1 \leq k \leq n-1$, $O \se \H^n$ be an open set and $S \se \he^n$ be an orientable $k$-codimensional $C^1_\H$ submanifold with boundary in $O$. Fix a choice of a continuous $(2n+1-k)$-vector field $t^\H_S$ on $O \cap S$ orienting $O \cap S$ and  extend $t^\H_S$  continuously on $O \cap \overline{S}$. Let $\nu_{\pp S}$ be the unit outward normal to $\pp S$. We define the \emph{induced orientation on the boundary} as the one fixed by the continuous $(2n-k)$-vector field $t^\H_{\pp S}$ on $O \cap \pp S$ such that
\[
\nu_{\pp S} \wedge t^\H_{\pp S}=t^\H_{S}\qquad \text{on }O \cap \pp S.
\]
\end{defi}

We now turn to the case of critical codimension. In this case the role of the unit outward normal $\nu_{\pp S}$ will be played, in some sense, either by the constant vector field $T$ or by the constant vector field $-T$. We recall that $S$ can either {\em lie above} or {\em lie below its boundary} according to the terminology introduced in Remark~\ref{rem_soprasotto}.

\begin{defi}\label{def_indcr}
Let $O \se \H^n$ be an open set and $S \se \he^n$ be an orientable $n$-codimensional $C^1_\H$ submanifold with boundary in $O$. Fix a choice of a continuous $(n+1)$-vector field $t^\H_S$ on $O \cap S$ determining an orientation on $O \cap S$, that we extend continuously on $O \cap \overline{S}$. We define the \emph{induced orientation on the boundary} $\pp S$ as the one determined by the  continuous horizontal unit $n$-vector $\tau_{\pp S}$ tangent to $\pp S$  such that the following holds for every $p\in O\cap \pp S$: if  $S$ lies above its boundary around $p$, then
\[
-T \wedge \tau_{\pp S} (p)=t^\H_{S}(p),
\]
while if $S$ lies below its boundary around $p$ then
\[
T \wedge \tau_{\pp S}(p)=t^\H_{S}(p).
\]
\end{defi}

\section{Approximation of submanifolds with boundary} \label{sec_approssimazione} 
In this section we show that submanifolds with boundary can be locally approximated by $C^1$ submanifolds with boundary. The approximation is in the sense of weak convergence of currents, which are locally well defined because every submanifold is locally orientable; however, the careful reader will notice that the convergence is also in the sense of local Hausdorff distance and of the $C^1_\H$ topology. Note also that, because of Lemma \ref{lem_lowdimsubwbound}, low dimensional $C^1_\H$-submanifold with boundary are Euclidean $C^1$-submanifold with boundary; therefore, in this section we deal only with low codimensional and critical codimensional $C^1_\H$-submanifold with boundary. 
We start with the case of non-critical codimension $k\leq n-1$.

\begin{lem}
\label{approx2}
Let   $1 \leq k \leq n-1$ be an integer,  $O \subset \he^n$  an open set, $S$  an oriented $k$-codimensional $C^1_\he$ submanifold  with boundary in $O$ such that $S,\pp S$ have locally finite measures\footnote{The measures on $S,\pp S$  are of course the Hausdorff measures on $S,\pp S$ of the appropriate dimensions that are (respectively) $2n+1-k,2n-k$ (in the low dimensional case $k\geq n+1$), $Q-k,Q-k-1$ (in the low-codimensional and non-critical case $k\leq n-1$) or $n+2,n$ (in the critical case $k=n$). } on $O$. Then for each $p \in O \cap \pp S$ there exist an open set $U$ with $p \in U$ and a sequence $(S_h)_{h \in \mathbb{N}}$ of oriented smooth and $C^1_\he$ regular submanifolds of codimension $k$ with smooth and $C^1_\he$ boundary in $O$ and locally finite measures in $O$ such that 
\begin{equation}\label{eq_AAA}
\curr{S_h}(\omega) \xrightarrow{h \to +\infty} \curr{S}(\omega)\qquad\text{for every $\omega \in \mathcal{D}^{2n+1-k}_\he(U)$}
\end{equation}
and
\begin{equation}\label{eq_BBB}
\curr{\pp S_h}(\alpha) \xrightarrow{h \to +\infty} \curr{\pp S}(\alpha)\qquad\text{for every $\alpha \in \mathcal{D}^{2n-k}_\he(U)$}.
\end{equation}
\end{lem}
\begin{proof}
By Theorem~\ref{teo_equivsupbound}, for every $p \in \pp S \cap O$ there exist a neighbourhood $U \subset O$ and $f_1,\dots,f_{k+1}\in C^1_\H(U)$ such that, writing $f=(f_1,\dots,f_{k+1})$ and $\hat f=(f_1,\dots,f_k)$,
\begin{align*}
& \nabla_\he f_1,\dots,\nabla_\he f_{k+1}\text{ are linearly independent in $U$ }\\
& U\cap S=\{q\in U : \hat f(q)=0, f_{k+1}(q)>0\}\\
& U\cap \pp S=\{q\in U : f(q)=0\}.
\end{align*}
We also introduce the $C^1_\H$-regular submanifold $S'=\{ q\in U :f_1(q)=\dots=f_k(q)=0\}\supset \overline S$. By Remark~\ref{rem_wlogCinfty} we can also assume that $f_1,\dots, f_{k+1}$ are smooth on $U\setminus S'$; we define
\begin{align*}
& S_h' \ceq \{q\in U : \hat f(q)=(\tfrac 1h,0,\dots,0)\}\\
& S_h \ceq \{q\in U : \hat f(q)=(\tfrac 1h,0,\dots,0),f_{k+1}(q)>0\}.
\end{align*}
We observe that $S_h$ are $k$-codimensional smooth and $C^1_\he$-regular submanifolds with smooth and $C^1_\he$-regular boundary $\pp S_h= \{q\in U : \hat f(q)=(\tfrac 1h,0,\dots,0),\ f_{k+1}(q)=0\}$. 

We claim that, possibly reducing $U$,~\eqref{eq_AAA} holds.
By Theorem~\ref{fond1} and the Implicit Function Theorem~\cite[Lemma~2.10]{jnv}, up to reducing $U$ there exist an orthogonal splitting $\he^n=\W\cdot\V$, an open subset $E\subset\W$ and  maps $\phi,\phi_h:E\to\V$ such that $S'=gr_\phi$, $S_h'=gr_{\phi_h}$ and $\phi_h$ are smooth. Let $E_h,E_\infty\subset E$ be such that $S_h=gr_{\phi_h|E_h}$ and $S=gr_{\phi|E_\infty}$; notice that $E_h,E_\infty$ are open. 
Reasoning as in the proof of Lemma~\ref{approx} we observe that, possibly reducing $E$ and $U$, for every $\omega\in \mathcal{D}^{2n+1-k}_\he(U)$
\begin{itemize}
\item[(a)] $\Phi_h \xrightarrow{h \to +\infty} \Phi$  uniformly on $E$, where $\Phi(\xi):=\xi \cdot \phi(\xi)$ and $\Phi_h(\xi):=\xi \cdot \phi_h(\xi)$
\item[(b)] 
$
\left( \nabla_\he \left(f_1-\frac{1}{h}\right) \wedge  \cdots \wedge \nabla_\he f_k \right) \circ \Phi_h \xrightarrow{h \to +\infty}\left( \nabla_\he f_1  \wedge \cdots \wedge \nabla_\he f_k \right) \circ \Phi
$
uniformly on $E$
\item[(c)]  $t^\he_{S_h} \circ \Phi_h \xrightarrow{h \to +\infty}t^\he_{S} \circ \Phi$ uniformly on $E$ 
\item[(d)]  $\omega \circ \Phi_h \xrightarrow{h \to +\infty} \omega \circ \Phi$ uniformly on $E$
\item[(e)]  $\Delta_h \circ \Phi \xrightarrow{h \to +\infty} \Delta \circ \Phi$ uniformly on $E$, where $\Delta,\Delta_h $ are as in Lemma~\ref{approx}
\item[(f)] there exists $C>0$ such that $\Delta_h\geq C$ and $\Delta\geq C$.
\end{itemize}
Moreover, we claim that
\begin{itemize}
\item[(g)] $\chi_{E_h}\to\chi_{E_\infty}$ almost everywhere on $E$. 
\end{itemize}
Statement (g) can be proved observing that
\begin{itemize}
\item if $\xi\in E_\infty$, then $f_{k+1}(\Phi_h(\xi))\to f_{k+1}(\Phi(\xi))>0$, therefore $\xi\in E_h$ for large $h$
\item $\mathcal H^{2n+1-k}_E(\pp E_\infty)=0$; this follows from~\eqref{eq_formulaareaFSSC} and taking into account that $\pp E_\infty=\Phi^{-1}(\pp S)$ and $\s^{Q-k}(\pp S)=0$
\item if $\xi\in E\setminus\overline{E_\infty}$, then $f_{k+1}(\Phi_h(\xi))\to f_{k+1}(\Phi(\xi))<0$, therefore $\xi\notin E_h$ for large $h$.
\end{itemize}
Using (a)--(g) we deduce
\begin{align*}
\lim_{h \to +\infty}\curr{S_h}(\omega)=
& \lim_{h \to +\infty}\int_{E_h} \langle t^\he_{S_h}(\Phi_h(\xi))|\omega(\Phi_h(\xi))\rangle \frac{|\nabla_\he \left(f_1-\frac{1}{h}\right)\wedge \cdots \wedge \nabla_\he f_k |}{\Delta_h}(\Phi_h(\xi))d\xi 
\\
=& \int_{E_\infty} \langle t^\he_{S}(\Phi(\xi))|\omega(\Phi(\xi))\rangle \frac{|\nabla_\he f_1\wedge \cdots \wedge \nabla_\he f_k |}{\Delta}(\Phi(\xi))d\xi
\\
=&\frac{1}{C_{n,k}}\int_{S\cap U}\langle t^\he_{S}(q)|\omega(q) \rangle d\s^{Q-k}(q)\\
=&\curr{S}(\omega),
\end{align*}
as desired.

Statement~\eqref{eq_BBB} can be proved (possibly restricting $U$) with the same argument in Lemma~\ref{approx} by writing $\pp S,\pp S_h$ as intrinsic graphs of codimension $k+1$; we omit the boring details.
\end{proof}

We now turn to the case of critical codimension $k=n$.

\begin{lem}
\label{approx2_crit}
Let  $O \subset \he^n$ be an open set and $S$  an oriented $n$-codimensional $C^1_\he$ submanifold  with boundary in $O$ such that $S,\pp S$ have locally finite measures on $O$. Then, for each $p \in O \cap \pp S $ there exist an open set $U$ with $p \in U$ and a sequence $(S_h)_{h \in \mathbb{N}}$ of $C^1$ and $C^1_\he$ oriented regular submanifolds of codimension $n$ in $U$ with boundary $U \cap \pp S =U \cap \pp S_h$ for every $h \in \N$ and with locally finite measures in $O$ such that
\begin{equation}\label{eq_AAA_crit}
\curr{S_h}(\omega) \xrightarrow{h \to +\infty} \curr{S}(\omega)\qquad\text{for every $\omega \in \mathcal{D}^{n+1}_\he(U)$}.
\end{equation}
\end{lem}

\begin{proof}
Let $p\in\pp S$ be fixed; denote by $S'$ a $C^1_\he$ submanifold as in Definition~\ref{def_supbound}. We use an approximation argument which is split into 8 steps. 

\emph{Step 1:  $S'$ as the level set of a $C^1_\he$ regular function as well as a graph.}\\
By Lemma~\ref{fsmooth}, up to an isometry of $\he^n$ there exist a neighbourhood $U$ of $p$, a function $f:\he^n \to \R^n$ and $C_0>0,\delta_0>0$ such that
\begin{equation}
\label{propf}
\begin{cases}
f \in C^1_\he(\he^n,\R^n),\\
f \in C^\infty(\he^n \setminus \lbrace f=0 \rbrace,\R^n),\\
|\nabla_\he f|\leq C_0 \text{ on }\H^n,\\
\widehat\nabla_\he f(q)\geq \delta_0 \Id_{n \times n} \forall q \in \H^n \text{ in the sense of quadratic forms} ,\\
S' \cap U=\lbrace q \in U : f(q)=0 \rbrace.
\end{cases}
\end{equation}
We assume without loss of generality that $S'=\lbrace f=0 \rbrace\subset\he^n$; by \cite[Theorem 1.4]{vittone},  $S'$ is the entire intrinsic graph $S'=gr_\phi$ of a suitable $\phi:\W \to \V$ where
\[
\W \ceq \lbrace(0,y,t) \in \R^n \times \R^n \times \R \equiv \he^n \rbrace ,\quad
\V \ceq \lbrace(x,0,0) \in \R^n \times \R^n \times \R \equiv \he^n \rbrace.
\]
Moreover, possibly reducing $U$ we can assume that $U=U_\W \cdot U_\V$ for  suitable open sets $U_\W \se \W$ and $U_\V \se \V$.

\emph{Step 2:  $\pi(\pp S)$  as the level set of a $C^1$ function $\hat{g}$.}\\
By Lemma \ref{ghat}, possibly reducing $U$, there exist $\hat{g} \in C^1(\R^{2n},\R^n)$ and $\delta_1>0$ such that
\[
\begin{cases}
\hat{g} \in C^1(\R^{2n},\R^n),\\
|\nabla \hat{g}| \text{ is bounded on } \R^{2n},\\
\nabla \hat{g}\circ \pi=\nabla_\he f \text{ on } \Sigma=U \cap \pp S,\\
\col[\pp_{x_1}\hat{g}|\cdots|\pp_{x_n}\hat{g}](z)\geq \delta_1 \Id_{n \times n} \forall z \in \R^{2n} \text{ in the sense of quadratic forms},\\
\pi(\Sigma) \se \lbrace \hat{g}=0 \rbrace.
\end{cases}
\]

\emph{Step 3: Construction of a $C^1$ and $C^1_\he$ regular submanifold $\tilde{S}$ which contains $\Sigma=U \cap \pp S$.}\\
Define  $g:\he^n \to \R^{n}$ by $g(x,y,t) \ceq \hat{g}(x,y)$; let  $\tilde{S} \ceq \lbrace g=0 \rbrace$. By Step 2 there exist $C_1>0$ and $\delta_1>0$ such that
\begin{equation}
\label{propgh}
\begin{cases}
|\nabla_\he g|\leq C_1 \mbox{ on } \he^{n},\\
\Sigma=U\cap\pp S \se \tilde S,\\
\widehat\nabla_\he g(q)=\col[\pp_{x_1}\hat{g}|\cdots|\pp_{x_n}\hat{g}](q)\geq \delta_1 \Id_{n \times n} \:\forall q \in \he^n \text{ in the sense of quadratic forms},\\
\nabla_\he {g}=\nabla_\he f \mbox{ on } \Sigma.
\end{cases}
\end{equation}
Using \cite[Theorem 1.4]{vittone} and the fact that  $g \in C^1(\he^{n},\R^n)$ one obtains
\begin{equation*}
\begin{cases}
\tilde{S} \mbox{ is a } C^1_\he \mbox{ regular submanifold and } \mbox{ Euclidean  $C^1 $ submanifold},\\
U \cap \tilde{S} = U\cap(\pi(\Sigma) \times \R) \subset \R^{2n}\times \R^n \equiv \he^n,\\
\tilde{S}=gr_{\tilde{\phi}} \mbox{ for some globally defined } \tilde{\phi} \in C^1(\W,\V).
\end{cases}
\end{equation*}

\emph{Step 4: Construction of  $C^1_\he$ regular submanifolds $S_h'$ approximating $S'$.}\\
We fix a non-negative kernel $K \in C^\infty_c(U(0,1))$ such that $\int K d\mathcal{L}^{2n+1}$ and we define
\[
K_r(p) \ceq \frac{1}{r^{2n+2}}K(\delta_{1/r}(p)).
\]
For every $h \in \N$ we also define  $\psi_h \in C^\infty_c(\Sigma_{3/h})$ (where by $\Sigma_\ve$ we denote the $\ve$-neighbourhood of $\Sigma$) such that $\psi_h \equiv 1$ on $\Sigma_{2/{h}}$ and $0 \leq \psi_h \leq 1$. Define
\[
\tilde{f}_h \ceq \psi_h(f \star K_{r_h})+(1-\psi_h)f,
\]
where the positive parameters $r_h\leq 1/h$ will be fixed later. Observe that
\[
\nabla_\he \tilde{f}_h=(\nabla_\he \psi_h)(f \star K_{r_h}-f)+\psi_h(\nabla_\he f \star K_{r_h})+(1-\psi_h)\nabla_\he f.
\]
Since $(f \star K_{r}-f)\xrightarrow{r}0$ uniformly on $U$, using~\eqref{propf} and choosing $r_h$ small enough one gets
\begin{equation}\label{eq_blablabla}
\begin{cases}
|\nabla_\he \tilde{f}_h|\leq 2C_0 +1, \\
\col[X_1\tilde{f}_h|\cdots|X_n\tilde{f}_h](q)\geq \frac{\delta_0}{2}\Id_{n \times n}\:\forall q \in \he^n \text{ in the sense of quadratic forms},\\
\| f \star K_{r_h}-f\|_{C^0(U)}\leq \frac{1}{h^2}.
\end{cases}
\end{equation}
Since $\nabla_\he g=\nabla_\he f \mbox{ on }\Sigma$ we have
\[
\| \nabla_\he (g-f)\|_{C^0(\Sigma_{3/h})}=o(1) \mbox{ as }h \to +\infty;
\]
this, together with the fact that $g = f=0$ on $\Sigma$, implies that
\[
\|g-f\|_{C^0(\Sigma_{3/h})}=o\left(\tfrac{1}{h}\right) \mbox{ as }h \to +\infty.
\]
Using~\eqref{eq_blablabla}, for small enough $r_h$ we find
\begin{equation}
\label{propft}
\|g-\tilde{f}_h\|_{C^0(\Sigma_{2/h})}=o\left(\tfrac{1}{h}\right).
\end{equation}
For every $h \in \N$ we fix a cut-off function $\chi_h \in C^\infty_c(\Sigma_{2/{h}})$ such that
\begin{equation}
\label{propchi}
\begin{cases}
0 \leq \chi_h \leq 1,\\
\chi_h \equiv 1 \mbox{ on }\Sigma_{1/h},\\
|\nabla_\he \chi_h|\leq C_2h,
\end{cases}
\end{equation}
for a suitable $C_2>0$ not depending on $j$. Let $v_h \ceq (\tfrac{2}{h^2},0,...,0)\in \R^n$ and define
\[
f_h \ceq \chi_hg+(1-\chi_h)(\tilde{f}_h-v_h).
\]
We observe that
\[
\nabla_\he f_h=\nabla_\he \chi_h(g-\tilde{f}_h+v_h)+\chi_h \nabla_\he g+(1-\chi_h)\nabla_\he \tilde{f}_h.
\]
Because of \eqref{propft} and \eqref{propchi} one has
\[
\|\nabla_\he \chi_h(g-\tilde{f}_h+v_h)\|_{C^0(\Sigma_{2/h})} \leq C_2h \left( o\left(\frac{1}{h}\right)+\frac{1}{h^2} \right)=o(1) 
\]
while 
\[
\nabla_\he \chi_h(g-\tilde{f}_h+v_h) \equiv 0 \mbox{ out of }\Sigma_{2/h}.
\]
Therefore, if $h$ is large enough we get by \eqref{propgh} and~\eqref{eq_blablabla} that for suitable $C_3>0$ and $\delta_3>0$
\begin{equation}\label{eq_propfj}
\begin{cases}
|\nabla_\he f_h|\leq C_3 \mbox{ on }\he^n,\\
\col[X_1f_h|\cdots|X_nf_h](q)\geq \delta_3 \Id_{n \times n} \:\forall q \in \he^n \text{ in the sense of quadratic forms}.
\end{cases}
\end{equation}
Therefore, the level set $S'_h \ceq \lbrace f_h=0 \rbrace$  is a $C^1_\he$ regular submanifold and, by \cite[Theorem 1.4]{vittone}, it is also the intrinsic Lipschitz graph of a globally defined function $\phi_h:\W \to \V$.

\emph{Step 5: $U \cap S'_h$ is a  Euclidean $C^1$ submanifold and $\phi_h\in C^1(U_\W,\V)$.}\\
Let $q \in S'_h$, i.e., $f_h(q)=0$. If $q \in \Sigma_{{2}/{h}}$, then $\psi_h \equiv 1$ in a neighbourhood of $q$ and (again in a neighbourhood of $q$)
\[
f_h=\chi_hg+(1-\chi_h)(f \star K_{r_h}-v_h).
\]
In particular, $f_h$ is  $C^1$ regular in a neighbourhood of $q$, hence $S'_h$ is of class $C^1$ in a neighbourhood of $q$. Instead, if $q \in U \setminus \Sigma_{2/h}$, then $\chi_h\equiv 0$ in a neighbourhood of $q$ and 
\[
0=f_h(q)=\tilde{f}_h(q)-v_h=\psi_h(q)(f \star K_{r_h})(q)+(1-\psi_h(q))f(q)-v_h,
\]
i.e.,
\[
f(q)=v_h+\psi_h(q)\Big(f(q)-(f \star K_{r_h})(q)\Big).
\]
Using~\eqref{eq_blablabla}
\[
|f(q)|\geq \|v_h\|-\|f-f\star K_{r_h}\|_{C^0(U)}\geq \frac{2}{h^2}-\frac{1}{h^2}>0,
\]
hence $f(q) \neq 0 $ and, by~\eqref{propf}, $f$ is $C^\infty$ in a neighbourhood of $q$ where one also has
\[
f_h=\tilde{f}_h-v_h=\psi_h(f \star K_{r_h})+(1-\psi_h)f-v_h.
\]
It follows that $f_h$ is  are $C^\infty$ smooth in a neighbourhood of $q$; in particular, $U \cap S'_h$ is a $C^1$ Euclidean submanifold and $\phi_h$ is $C^1$ on $U_\W$.

\emph{Step 6: Construction of the approximating submanifolds $S_h$.}\\
Let $\pi_\W:\he^n \to \W$ be the  projection $\pi_\W(x,y,t) \ceq (x,y,t)\cdot (-x,0,0)$; in particular, $\pi_\W(U)=\pi_\W(U_\W \cdot U_\V)=U_\W$. Recall that $S \se S'=gr_\phi$, $S'_h=gr_{\phi_h}$ and $\tilde{S}=gr_{\tilde{\phi}}$. The open set $U^+_\W \ceq \pi_\W(U \cap S) \se U_\W$ satisfies
\[
U_\W \cap \pp U^+_\W=\pi_\W(\Sigma).
\]
The submanifolds $S_h$ defined for $h \in \N$ by $S_h \ceq gr_{\phi_h|U^+_\W}$ are relatively open subsets of $S_h'$, hence they are both  $C^1_\he$ regular and  Euclidean  $C^1$ submanifold.

\emph{Step 7: $U \cap  S_h$ is a $C^1$ manifold with boundary and $U \cap \pp S_h=U \cap \pp S$.}\\
For $h \in \N$ and $i=1,2,3$ we define $\Delta_h^i \ceq \pi_\W(\Sigma_{i/h})$; $\Delta_h^i$ are open neighbourhoods of $\pi_\W(\Sigma)$. We observe that $\phi_h \equiv \tilde{\phi}$ on $ \Delta^1_h$: this follows upon noticing that $f_h\equiv g $ on $\Sigma_{1/h}$, hence
\[
\Sigma_{1/h} \cap gr_{\phi_h}=\Sigma_{1/h}\cap S'_h=\Sigma_{1/h} \cap \tilde{S}=\Sigma_{1/h}\cap gr_{\tilde{\phi}}.
\]
This implies that $\phi_h \equiv \tilde{\phi}$ on $ \Delta^1_h$ as well as the fact that the $U\cap S_h$ is a classical $C^1$  submanifold with boundary $U \cap \pp S_h=U \cap \pp S$.

We observe in passing that
\begin{equation}
\label{eq_Delta3}
\text{$\phi_h \equiv \tilde{\phi}_h$ on $U^+_\W \setminus \Delta^3_h$, where $\tilde{\phi_h}\in C^\infty(\W,\V)$ is such that $\lbrace f=v_h \rbrace=gr_{\tilde{\phi_h}}$.}
\end{equation}

\emph{Step 8: $\curr{S_h}(\omega)\xrightarrow{h \to +\infty}\curr{S}(\omega)$ for every $\omega \in \mathcal{D}^{n+1}_\he(U)$.}\\
Let $\Phi_h(w) \ceq w \cdot \phi_h(w), w\in U_\W$, be the graph map associated with the $C^1$ map $\phi_h$. Let $\omega \in \mathcal{D}^{n+1}_\he(U)$ and $k \in \N$ be fixed. Later we will let $k \to +\infty$ after we let $h \to +\infty$ so we can suppose $h>k$. We use Theorem~\ref{fond1} to get
\[
\curr{S_h}(\omega)=\int_{U^+_\W} \langle t^\he_{S_h}(\Phi_h(\xi))|\omega(\Phi_h(\xi))\rangle \frac{|\nabla_\he (f_h)_1\wedge \cdots \wedge \nabla_\he (f_h)_n |}{|\det (\widehat{\nabla}_\he f_h)|}(\Phi_h(\xi))d\xi 
\]
where we used $(f_h)_l$ to denote the $l$-th component of $f_h$.
We split the integral as
\begin{align*}
\curr{S_h}(\omega)=&\int_{U^+_\W \setminus \Delta^3_k} \langle t^\he_{S_h}(\Phi_h(\xi))|\omega(\Phi_h(\xi))\rangle \frac{|\nabla_\he (f_h)_1\wedge \cdots \wedge \nabla_\he (f_h)_n |}{|\det (\widehat{\nabla}_\he f_h)|}(\Phi_h(\xi))d\xi \\
&+\int_{\Delta^3_k} \langle t^\he_{S_h}(\Phi_h(\xi))|\omega(\Phi_h(\xi))\rangle \frac{|\nabla_\he (f_h)_1\wedge \cdots \wedge \nabla_\he (f_h)_n |}{|\det (\widehat{\nabla}_\he f_h)|}(\Phi_h(\xi))d\xi
\end{align*}
We first focus on the second integral. Using the properties in \eqref{eq_propfj} we obtain that the area factor $\frac{|\nabla_\he (f_h)_1\wedge \cdots \wedge \nabla_\he (f_h)_n |}{|\det (\widehat{\nabla}_\he f_h)|}(\Phi_h(\xi))$ is bounded by a positive constant only depending on $\delta_3$ and $C_3$ while $|\langle t^\he_{S_h}(\Phi_h(\xi))|\omega(\Phi_h(\xi))\rangle|$ is bounded on $\Delta^3_k$ by $\|\omega\|_{C^0(\Sigma_{3/k})}<+\infty$. This means that
\[
\int_{\Delta^3_k} \langle t^\he_{S_h}(\Phi_h(\xi))|\omega(\Phi_h(\xi))\rangle \frac{|\nabla_\he (f_h)_1\wedge \cdots \wedge \nabla_\he (f_h)_n |}{|\det (\widehat{\nabla}_\he f_h)|}(\Phi_h(\xi))d\xi=O(\mathcal{L}^{n+1}(\Delta^3_k)).
\]
We now focus on the first integral. On $U \setminus \Sigma_{3/k}$ we recall that $f_h$ is defined (since $\psi_h \equiv 0$ on $U \setminus \Sigma_{3/k}$ because $h>k$) as 
\[
f_h=f-v_h.
\]
It follows that the horizontal derivatives of $f_h$ coincide with the horizontal derivatives of $f$ on $U \setminus \Sigma_{3/k}$. Since $\Phi_h \xrightarrow{h \to +\infty}\Phi$ uniformly on $U^+_\W \setminus \Delta^3_k$ the latter implies that 
\[
\frac{|\nabla_\he (f_h)_1\wedge \cdots \wedge \nabla_\he (f_h)_n |}{|\det (\widehat{\nabla}_\he f_h)|}\circ \Phi_h \xrightarrow{h \to +\infty}
\frac{|\nabla_\he f_1\wedge \cdots \wedge \nabla_\he f_n |}{|\det (\widehat{\nabla}_\he f)|}\circ \Phi \text{ uniformly on } U^+_\W \setminus \Delta^3_k
\]
(where with $f_l$, $1 \leq l \leq n$, we denoted the $l$-th component of $f$) and
\[
\langle t^\he_{S_h}\circ \Phi_h|\omega \circ \Phi_h \rangle \xrightarrow{h \to +\infty} \langle t^\he_{S}\circ \Phi|\omega \circ \Phi\rangle \text{ uniformly on } U^+_\W \setminus \Delta^3_k,
\]
in the same fashion as in the proof of Lemma~\ref{approx}. Then we obtain
\begin{align*}
&\lim_{h \to +\infty}\int_{U^+_\W \setminus \Delta^3_k} \langle t^\he_{S_h}(\Phi_j(\xi))|\omega(\Phi_h(\xi))\rangle \frac{|\nabla_\he (f_h)_1\wedge \cdots \wedge \nabla_\he (f_h)_n |}{|\det (\widehat{\nabla}_\he f_h)|}(\Phi_h(\xi))d\xi  \\
=\:& 
\int_{U^+_\W \setminus \Delta^3_k} \langle t^\he_{S}(\Phi(\xi))|\omega(\Phi(\xi))\rangle \frac{|\nabla_\he f_1\wedge \cdots \wedge \nabla_\he f_n |}{|\det (\widehat{\nabla}_\he f)|}(\Phi(\xi))d\xi \\=\:& \curr{S}(\omega)- \int_{\Delta^3_k} \langle t^\he_{S}(\Phi(\xi))|\omega(\Phi(\xi))\rangle \frac{|\nabla_\he f_1\wedge \cdots \wedge \nabla_\he f_n |}{|\det (\widehat{\nabla}_\he f)|}(\Phi_j(\xi))d\xi.
\end{align*}
As before we can estimate
\[
\int_{\Delta^3_k} \langle t^\he_{S}(\Phi(\xi))|\omega(\Phi(\xi))\rangle \frac{|\nabla_\he f_1\wedge \cdots \wedge \nabla_\he f_n |}{|\det (\widehat{\nabla}_\he f)|}(\Phi_h(\xi))d\xi=O(\mathcal{L}^{n+1}(\Delta^3_k))
\]
so that, for every $k\in\N$,
\[
\limsup_{h \to +\infty}|\curr{S_h}(\omega)-\curr{S}(\omega)| \leq O(\mathcal{L}^{n+1}(\Delta^3_k)).
\]
Letting $k \to +\infty$ and using the fact that $O(\mathcal{L}^{n+1}(\Delta^3_k))\to\mathcal{L}^{n+1}(\pi_\W(\Sigma))=0$ we obtain 
\[
\curr{S_h}(\omega)\xrightarrow{h \to +\infty}  \curr{S}(\omega)
\]
for every $\omega \in \mathcal{D}^{n+1}_\he(U)$.
\end{proof}

\begin{rem}\label{rem_miglioramentiCinfty}
By a standard approximating procedure, the submanifolds $S_h$  of Lemma~\ref{approx2_crit} can be chosen to with the extra property that each $S_h$ (to be understood without its boundary) is actually $C^\infty$ smooth.
\end{rem}

\section{Proof of Stokes' Theorem}
We are  ready to prove our main result. Observe that Theorem~\ref{thm_intro}  is an immediate consequence of the following theorem.

\begin{teo}
\label{fin1}
Let $1\leq k\leq 2n+1$ be an integer, $O \subset \he^n$ be an open set and $S \subset O$ be an orientable $k$-codimensional $C^1_\he$ regular submanifold with boundary\footnote{The careful reader will notice that Theorem~\ref{fin1} holds also when $O\cap\pp S$ is empty.} such  that $S,\pp S$ have locally finite  measures on $O$. Then
\[
\curr{S}(d_c\omega)=\curr{\pp S}(\omega)\qquad\text{for every }\omega \in \D^{2n-k}_\he(O)
\]
or, equivalently, $\pp_c\curr{S}=\curr{\partial S}$.
\end{teo}
\begin{proof}
The low-dimensional case, i.e., when  $k\geq n+1$, is an easy consequence of the classical Stokes' Theorem;  we then consider the low-codimensional case and assume hereafter that $k\leq n$.
We first define a suitable open covering of $O \cap \overline{S}$ as follows. For each $p  \in  S$ we use Lemma \ref{approx} to find an open neighbourhood $U^1_p\subset O$ and a sequence $(S_h)_{h \in \mathbb{N}}$ of smooth and $C^1_\he$ submanifolds of codimension $k$ contained in $U^1_p$  such that
\[
\curr{S_h}(\omega) \xrightarrow{h \to +\infty} \curr{S}(\omega)\qquad\text{for every $\omega \in \mathcal{D}^{2n+1-k}_\he(U_p^1)$.}
\] 
For each $p \in O \cap \pp S$ we can use Lemma~\ref{approx2} (if $k\leq n-1$) or Lemma~\ref{approx2_crit} (if $k=n$) to find an open neighbourhood $U^2_p$ of $p$ and a sequence $(S_h)_{h \in \mathbb{N}}$ of $C^1$ and $C^1_\he$ $k$-codimensional submanifolds with ($C^1$) boundary in $U^2_p$  such that
\begin{align*}
& \curr{S_h}(\omega) \xrightarrow{h \to +\infty} \curr{S}(\omega)\qquad\text{for every $\omega \in \mathcal{D}^{2n+1-k}_\he(U_p^2)$}\\
& \curr{\pp S_h}(\alpha) \xrightarrow{h \to +\infty} \curr{\pp S}(\alpha)\qquad\text{for every $\alpha \in \mathcal{D}^{2n-k}_\he(U_p^2)$}.
\end{align*}
We extract from the families $(U_p^1)_{p \in S}$ and $(U_p^2)_{p \in O\cap \pp S}$ countable (or, possibly, finite) sub-families  $(U^1_i)_{i \in \mathbb{N}}$ and  $(U^2_i)_{i \in \mathbb{N}}$ such that $O \cap \overline{S} \subset \bigcup_{i \in \mathbb{N}}U_i^1 \cup \bigcup_{i \in \mathbb{N}}U_i^2$. We also fix a partition of the unity, i.e., functions $\zeta_i^j\in C^\infty_c(U_i^j), i\in\N, j\in\{1,2\},$ such that 
\begin{equation}\label{eq_partizionedellunita}
0\leq\zeta_i^j\leq 1\qquad\text{and}\qquad\sum_{\substack{i \in \mathbb{N} \\ j=1,2}}\zeta_i^j=1\text{ on $\overline S$.}
\end{equation}
It is not restrictive to assume that the covering $(U_i^1)_i\cup (U_i^2)$ of $O\cap\overline S$ is locally finite, so that the sum in~\eqref{eq_partizionedellunita} is well defined.

Let $\omega \in \mathcal{D}^{2n-k}_\he(O)$ be fixed; for every $i \in \mathbb{N}$ and  $j=1,2$ denote by $(S_h)_h$  the $C^1$ submanifolds approximating  $S$ (in the sense of currents) in $U^j_i$. Using Lemma \ref{eqcurr}  we get
\begin{equation}\label{eq_elenco}
\begin{split}
\curr{S}(d_c(\zeta^j_i \omega)) 
&= \lim_{h\to\infty} \curr{S_h}(d_c(\zeta^j_i \omega))
 = \lim_{h\to\infty}  \int_{S_h}d_c(\zeta^j_i \omega)\\
&\stackrel{*}{=}\lim_{h\to\infty}  \int_{\pp S_h} \zeta^j_i \omega
 = \lim_{h\to\infty}\curr{\partial S_h}(\zeta^j_i \omega)\  =\ \curr{\partial S}(\zeta^j_i \omega),
\end{split}
\end{equation}
where the equality marked by $\ast$ follows from the classical Stokes' Theorem if $k\leq n-1$ (when $d_c=d$ is the classical exterior differentiation) while, if $k=n$, it follows from Stokes' Theorem and Remark~\ref{rem_defDRumin}  upon observing that
\begin{align*}
\int_{S_h}D(\zeta^j_i \omega)
&= \int_{S_h}d(\zeta^j_i \omega-\theta \wedge L^{-1}((d(\zeta^j_i \omega))_{\hel_1}))\\
&=\int_{\pp S_h}\zeta^j_i \omega-\theta \wedge L^{-1}((d(\zeta^j_i \omega))_{\hel_1}) =\int_{\pp S}\zeta^j_i\omega,
\end{align*}
the last equality following from the fact that $\pp S_h$ equals $\pp S$ and it is tangent to the horizontal distribution. Finally, 
\begin{align*}
\pp_c \curr{S}(\omega)&=\sum_{{\substack{i \in \mathbb{N} \\ j=1,2}}}\pp_c \curr{S}(\zeta^j_i \omega)=\sum_{{\substack{i \in \mathbb{N} \\ j=1,2}}} \curr{S}(d_c(\zeta^j_i \omega))=\sum_{{\substack{i \in \mathbb{N} \\ j=1,2}}}\curr{\partial S}(\zeta^j_i \omega)=\curr{\pp S}(\omega)
\end{align*}
and the proof is accomplished.
\end{proof}

\bibliographystyle{acm}
\bibliography{DMJNGVbib}

\begin{thebibliography}{10}

\bibitem{ASCV}
{\sc Ambrosio, L., Serra~Cassano, F., and Vittone, D.}
\newblock Intrinsic regular hypersurfaces in {H}eisenberg groups.
\newblock {\em J. Geom. Anal. 16}, 2 (2006), 187--232.

\bibitem{ArenaSerapioni}
{\sc Arena, G., and Serapioni, R.}
\newblock Intrinsic regular submanifolds in {H}eisenberg groups are
  differentiable graphs.
\newblock {\em Calc. Var. Partial Differential Equations 35}, 4 (2009),
  517--536.

\bibitem{pansu}
{\sc Balogh, Z., Kozhevnikov, A., and Pansu, P.}
\newblock {H}\"older maps from {E}uclidean spaces to {C}arnot groups, 2017.
\newblock Preprint available online at
  \url{https://www.imo.universite-paris-saclay.fr/~pierre.pansu/RnCarnot.pdf}.

\bibitem{BSC2}
{\sc Bigolin, F., and Serra~Cassano, F.}
\newblock Distributional solutions of {B}urgers' equation and intrinsic regular
  graphs in {H}eisenberg groups.
\newblock {\em J. Math. Anal. Appl. 366}, 2 (2010), 561--568.

\bibitem{BSC1}
{\sc Bigolin, F., and Serra~Cassano, F.}
\newblock Intrinsic regular graphs in {H}eisenberg groups vs. weak solutions of
  non-linear first-order {PDE}s.
\newblock {\em Adv. Calc. Var. 3}, 1 (2010), 69--97.

\bibitem{BV}
{\sc Bigolin, F., and Vittone, D.}
\newblock Some remarks about parametrizations of intrinsic regular surfaces in
  the {H}eisenberg group.
\newblock {\em Publ. Mat. 54}, 1 (2010), 159--172.

\bibitem{bryant}
{\sc Bryant, R., Griffiths, P., and Grossman, D.}
\newblock {\em Exterior differential systems and {E}uler-{L}agrange partial
  differential equations}.
\newblock Chicago Lectures in Mathematics. University of Chicago Press,
  Chicago, IL, 2003.

\bibitem{Canarecci_orientabilita}
{\sc Canarecci, G.}
\newblock Notion of {$\Bbb H$}-orientability for surfaces in the {H}eisenberg
  group {$\Bbb H^n$}.
\newblock {\em Differential Geom. Appl. 74\/} (2021), Paper No. 101701, 19.

\bibitem{CanarecciCurrents}
{\sc Canarecci, G.}
\newblock Sub-{R}iemannian currents and slicing of currents in the {H}eisenberg
  group {$\Bbb{H}^n$}.
\newblock {\em J. Geom. Anal. 31}, 5 (2021), 5166--5200.

\bibitem{CittiMan}
{\sc Citti, G., and Manfredini, M.}
\newblock Implicit function theorem in {C}arnot-{C}arath\'{e}odory spaces.
\newblock {\em Commun. Contemp. Math. 8}, 5 (2006), 657--680.

\bibitem{Corni1}
{\sc Corni, F.}
\newblock Intrinsic regular surfaces of low codimension in {H}eisenberg groups.
\newblock {\em Ann. Fenn. Math. 46}, 1 (2021), 79--121.

\bibitem{CorniMagnani}
{\sc Corni, F., and Magnani, V.}
\newblock Area formula for regular submanifolds of low codimension in
  {H}eisenberg groups.
\newblock {\em Adv. Calc. Var. 16}, 3 (2023), 665--688.

\bibitem{cornimagnani2}
{\sc Corni, F., and Magnani, V.}
\newblock Area of intrinsic graphs in homogeneous groups, 2023.
\newblock Preprint available online at \url{https://arxiv.org/abs/2311.06638}.

\bibitem{didonatophd}
{\sc Di~Donato, D.}
\newblock Intrinsic differentiability and {I}ntrinsic {R}egular {S}urfaces in
  {C}arnot groups, 2017.
\newblock PhD Thesis available online at
  \url{http://eprints-phd.biblio.unitn.it/2660/1/TesiFinaleDottorato_DiDonatoDaniela.pdf}.

\bibitem{DDFO}
{\sc Di~Donato, D., F\"{a}ssler, K., and Orponen, T.}
\newblock Metric rectifiability of {$\Bbb{H}$}-regular surfaces with
  {H}\"{o}lder continuous horizontal normal.
\newblock {\em Int. Math. Res. Not. IMRN}, 22 (2022), 17909--17975.

\bibitem{evansgariepy}
{\sc Evans, L.~C., and Gariepy, R.~F.}
\newblock {\em Measure theory and fine properties of functions}, revised~ed.
\newblock Textbooks in Mathematics. CRC Press, Boca Raton, FL, 2015.

\bibitem{FischerTripaldi}
{\sc Fischer, V., and Tripaldi, F.}
\newblock An alternative construction of the {R}umin complex on homogeneous
  nilpotent {L}ie groups.
\newblock {\em Adv. Math. 429\/} (2023), Paper No. 109192, 39.

\bibitem{follandstein}
{\sc Folland, G.~B., and Stein, E.~M.}
\newblock {\em Hardy spaces on homogeneous groups}, vol.~28 of {\em
  Mathematical Notes}.
\newblock Princeton University Press, Princeton, NJ; University of Tokyo Press,
  Tokyo, 1982.

\bibitem{franchibis}
{\sc Franchi, B., Serapioni, R., and Serra~Cassano, F.}
\newblock Rectifiability and perimeter in the {H}eisenberg group.
\newblock {\em Math. Ann. 321}, 3 (2001), 479--531.

\bibitem{franchi}
{\sc Franchi, B., Serapioni, R., and Serra~Cassano, F.}
\newblock Regular submanifolds, graphs and area formula in {H}eisenberg groups.
\newblock {\em Adv. Math. 211}, 1 (2007), 152--203.

\bibitem{franchitris}
{\sc Franchi, B., Tchou, N., and Tesi, M.~C.}
\newblock Div-curl type theorem, {$H$}-convergence and {S}tokes formula in the
  {H}eisenberg group.
\newblock {\em Commun. Contemp. Math. 8}, 1 (2006), 67--99.

\bibitem{jnv}
{\sc Julia, A., Nicolussi~Golo, S., and Vittone, D.}
\newblock Area of intrinsic graphs and coarea formula in {C}arnot groups.
\newblock {\em Math. Z. 301}, 2 (2022), 1369--1406.

\bibitem{JNGVIMRN}
{\sc Julia, A., Nicolussi~Golo, S., and Vittone, D.}
\newblock Lipschitz functions on submanifolds of {H}eisenberg groups.
\newblock {\em Int. Math. Res. Not. IMRN}, 9 (2023), 7399--7422.

\bibitem{LerarioTripaldi}
{\sc Lerario, A., and Tripaldi, F.}
\newblock Multicomplexes on {C}arnot groups and their associated spectral
  sequence.
\newblock {\em J. Geom. Anal. 33}, 7 (2023), Paper No. 199, 22.

\bibitem{magnani}
{\sc Magnani, V.}
\newblock Towards a theory of area in homogeneous groups.
\newblock {\em Calc. Var. Partial Differential Equations 58}, 3 (2019), Paper
  No. 91, 39.

\bibitem{RuminCR}
{\sc Rumin, M.}
\newblock Un complexe de formes diff\'{e}rentielles sur les vari\'{e}t\'{e}s de
  contact.
\newblock {\em C. R. Acad. Sci. Paris S\'{e}r. I Math. 310}, 6 (1990),
  401--404.

\bibitem{rumin}
{\sc Rumin, M.}
\newblock Formes diff\'{e}rentielles sur les vari\'{e}t\'{e}s de contact.
\newblock {\em J. Differential Geom. 39}, 2 (1994), 281--330.

\bibitem{tripaldi2020rumin}
{\sc Tripaldi, F.}
\newblock The {R}umin complex on nilpotent {L}ie groups, 2020.
\newblock Preprint available online at \url{https://arxiv.org/abs/2009.10154}.

\bibitem{vittone}
{\sc Vittone, D.}
\newblock Lipschitz graphs and currents in {H}eisenberg groups.
\newblock {\em Forum Math. Sigma 10\/} (2022), Paper No. e6, 104.

\end{thebibliography}

\end{document}